\documentclass[12pt]{amsart}
\usepackage{amsmath,amsfonts,amssymb,amscd,amsthm,amsbsy,epsf}
\newtheorem{thm}{Theorem}[section]
\newtheorem*{thm1}{Theorem 3.18}

\newtheorem{lem}[thm]{Lemma}
\newtheorem{cor}[thm]{Corollary}
\newtheorem{prop}[thm]{Proposition}
\theoremstyle{definition} 		%% bold "label" with roman text
\newtheorem{defn}[thm]{Definition}
\newtheorem{remark}[thm]{Remark}
\def\nat{{\mathbb N}}
\def\A{{\mathcal A}}
\def\B{{\mathcal B}}
\def\C{{\mathcal C}}

\def\F{{\mathcal F}}
\def\G{{\mathcal G}}

\def\R{{\mathcal R}}
\def\U{{\mathcal U}}

\def\bs{{\mathbf s}}
\def\bt{{\mathbf t}}
\def\bu{{\mathbf u}}
\def\bw{{\mathbf w}}

\def\Tau{{\mathfrak T}}

\begin{document}
\title{Ramsey theory for words over an infinite alphabet}
\author{Vassiliki  Farmaki }

\begin{abstract}
A complete partition theory is presented for $\omega$-located words (and $\omega$-words), namely for located words $w=w_{n_1}\ldots w_{n_l}$, over an infinite alphabet $\Sigma=\{\alpha_1, \alpha_2, \ldots \}$, such that $ w_{n_i}\in \{\alpha_1, \alpha_2, \ldots, \alpha_{k_{n_i}} \}$ for every $1\leq i \leq l$, where $\vec{k}=(k_n)_{n\in\nat}\in \nat^{\nat}$ is a fixed increasing sequence. This theory strengthens in an essential way the classical Carlson, Furstenberg-Katznelson, and Bergelson-Blass-Hindman partition Ramsey theory for words over a finite alphabet. Consequences of this theory are strong simultaneous extensions of the classical Hindman, Milliken-Taylor partition theorems, and of a van der Waerden theorem  for general semigroups, extending results of Hindman-Strauss and Beiglb$\ddot{o}$ck.
\end{abstract}


\maketitle
\baselineskip=18pt
\pagestyle{plain}              

%%%%%%   BODY   %%%%%%%%%%%%%%%%%%%%%%%%%%%%%%

\section*{Introduction}

The concept of a word over a finite alphabet was introduced in Ramsey theory by Hales-Jewett \cite{HaJ}, providing a purely combinatorial proof of van der Waerden's theorem \cite{vdW} on the existence of arbitralily long arithmetic progressions in one of the cells of any partition of the positive integers. Subsequently, words over a finite alphabet, in the work of Carlson \cite{C} and Furstenberg-Katznelson \cite{FuK}, proved an essential tool for the unification of the two branches of Ramsey theory, the one involving Ramsey's classical theorem \cite{R} and Nash-Williams type partition theorem \cite{NW}, and its extensions by Hindman \cite{H} and Milliken \cite{M}-Taylor \cite{T}, the other the van der Waerden and the Hales-Jewett theorems, just mentioned. These tools were extended and strengthened, with the systematic introduction of Schreier sets, in \cite{FN1}, \cite{FN2}. 

The concept of a located word over a finite alphabet $\Sigma$ introduced formally by Bergelson-Blass-Hindman in \cite{BBH} as a function from a finite subset of the set $\nat$ of natural numbers, the support and location of the word, into the alphabet $\Sigma$. They established a partition theorem for located words over a finite alphabet and also a Ramsey type and a Nash-Williams type partition theorem for located words over a finite alphabet.

In all these results, the alphabet $\Sigma$ under consideration is always assumed to be \textbf{finite}. It is clear that these combinatorial results do not generally hold if $\Sigma$ is assumed to be infinite, since e.g. it is generally impossible to find a monochromatic infinite arithmetic progression for every finite coloring of the set of natural numbers. However, it might still be possible to relax the strict finiteness condition for the alphabet. 

In the present work we introduce a relaxation as follows: we start with an \textbf{infinite} alphabet $\Sigma=\{\alpha_1, \alpha_2, \ldots \}$, ordered according to the natural numbers, and a sequence $\vec{k}=(k_n)_{n\in\nat}$ of positive integers, the \textbf{dominating} sequence, and we define an $\omega$-located word over $\Sigma$ dominated by $\vec{k}$ to be a located word $w=w_{n_1}\ldots w_{n_l}$ over $\Sigma$ such that in addition $ w_{n_i}\in \{\alpha_1, \alpha_2, \ldots, \alpha_{k_{n_i}} \}$ for every $1\leq i \leq l$. Similarly a word $w=w_{1}\ldots w_{l}$ over $\Sigma$ is an $\omega$-word over $\Sigma$ dominated by $\vec{k}$ if $ w_{i}\in \{\alpha_1, \alpha_2, \ldots, \alpha_{k_{i}} \}$ for every $1\leq i \leq l$. Thus words and located words are $\omega$-words and $\omega$-located words, in case the dominating sequence $(k_n)_{n\in\nat}$ is a constant sequence.

It turns out that the whole of infinitary Ramsey theory can be obtained for an infinite alphabet under the condition of functional domination. We thus obtain:

(a) Partition theorems for $\omega$-located words over an 
infinite countable alphabet (in Theorems~\ref{thm:block-Ramsey1}, ~\ref{thm:block-Ramsey2}, Corollary~\ref{central}), as well as partition theorems for (unlocated) $\omega$-words (in Theorem~\ref{thm:block-Ramsey3}) providing proper extensions of Bergelson-Blass-Hindman's partition theorem (Theorem 4.1 in \cite{BBH}) for located words over a finite alphabet and Carlson's partition theorem for words over a finite alphabet (Lemma 5.9 in \cite{C}) respectively. As  consequences of these theorems, 
we prove partition theorems for semigroups (Corollary~\ref{thm:van der Waerden}), including the results of Hindman and Strauss in \cite{HS} (Theorems 14.12, 14.15), which are simultaneous extensions of the Hindman \cite{H} and the van der Waerden \cite{vdW} partition theorems.

(b) Extended Ramsey type partition theorems for each countable ordinal $\xi$ for variable and constant $\omega$-located words (in Theorems~\ref{thm:block-Ramsey} and ~\ref{thm:block-Ramsey8}), involving the $\xi$-Schreier sequences of $\omega$-located words, which results imply an ordinal extension of Bergelson-Blass-Hindman's Ramsey type partition theorem for words over a finite alphabet (Theorem 5.1 in  \cite{BBH}), corresponding to the case of finite ordinals, and strengthen Furstenberg-Katznelson's Ramsey type partition theorem (Theorems 2.7 and 3.1 in \cite{FuK}) for words over a finite alphabet 
(Theorem~\ref{cor:k-Ramsey words}). Furthermore consequences of Theorem~\ref{thm:block-Ramsey} are multidimentional partition theorems for semigroups corresponding to each countable order $\xi$ 
(Corollaries~\ref{cor:multi}, ~\ref{thm:van der Waerden1}) providing strong simultaneous extension of the block-Ramsey partition theorem for every countable ordinal, proved in \cite{FN1}, and van der Waerden's theorem in \cite{vdW} and extending the partition theorem of Beiglb$\ddot{o}$ck in \cite{Be} for commutative semmigroups, corresponding to the case of finite ordinals. 

(c) A partition theorem for infinite orderly sequences of variable $\omega$-located words (Theorem~\ref{cor:blockNW}), which can be said to be a Nash-Williams type partition theorem for variable $\omega$-located words, strengthening and extending the partition theorem for infinite sequences of variable located words over a finite alphabet proved in  \cite{BBH}(Theorem 6.1) and also Carlson's partition theorem (Theorem 2 in \cite{C}) for infinite sequences of variable words over a finite alphabet.  
As a consequence of Theorem~\ref{cor:blockNW} we prove, in Theorems~\ref{cor:centralNW} and ~\ref{cor:ncentralNW}, partition theorems for infinite sequences in a commutative and in a noncommutative semigroup, respectively which are strong simultaneous extensions of the infinitary partition theorem of Milliken \cite{M}, Taylor \cite{T} and van der Waerden \cite{vdW} applied to semigroups. In order to state Theorem~\ref{cor:centralNW} below, let $\nat=\{1,2,\ldots\}$ be the set of positive integers and for a sequence $(x_n)_{n\in\nat}$ in a semigroup $(X,+)$ let 
\begin{center}
$FS\big((x_n)_{n\in\nat}\big)=\{x_{n_1}+\ldots+x_{n_\lambda} : \lambda\in\nat, n_1<\cdots<n_\lambda\in\nat\}$, and
\end{center}
\begin{center}
$\big[FS\big((x_n)_{n\in\nat}\big)\big]^\omega = \{(y_n)_{n\in\nat} : y_n\in FS\big((x_n)_{n\in\nat}\big)$ and $y_n<y_{n+1}$ for every $n\in\nat\}$.
\end{center}
For $y=x_{n_1}+\ldots+x_{n_\lambda}, z=x_{m_1}+\ldots+x_{m_\nu}\in FS\big((x_n)_{n\in\nat}\big)$ we write $y<z$ if $n_\lambda<m_1$.

\begin{thm1}
Let $(X,+)$ be a commutative semigroup, $\vec{k}=(k_n)_{n \in \nat}\subseteq \nat$ an increasing sequence and $(y_{l,n})_{n \in \nat}$ for every $l \in \nat,$ sequences in $X$. If $\U \subseteq X^{\nat}$ is a pointwise closed family of $X^{\nat}$, then there exist sequences $(E_n)_{n \in \nat}$ and $(H_n)_{n \in \nat}$ of non-empty finite subsets of $\nat$ satisfying $\max E_n<\min E_{n+1}$, $H_n\subseteq E_{n}$ for every $n \in \nat$ and a sequence $(\beta_n)_{n \in \nat}\subseteq X$, where $\beta_n=\sum_{j\in E_n \setminus H_n}y_{l^n_j,j}$ for $1\leq l^n_j \leq k_j$, such that for every function $f: \nat\rightarrow \nat$ with $f(n)\leq k_n$ for every $n\in\nat$
\newline
either  $\big[ FS\big((\beta_{n}+\sum_{t \in H_{n}}y_{f(n), t})_{n\in\nat}\big)\big]^\omega\subseteq \U$, 
or $\big[ FS\big((\beta_{n}+\sum_{t \in H_{n}}y_{f(n), t})_{n\in\nat}\big)\big]^\omega\subseteq X^{\nat}\setminus\U$.
\end{thm1}

(d) An Ellentuck type characterization of completely Ramsey partitions of the set of infinite orderly sequences of $\omega$-located words (in Theorem~\ref{cor:Baireproperty}).

The essentially stroger nature of this Ramsey theory for functionally dominated words over an infinite alphabet developed in this paper makes it reasonable to expect that will find substantial applications in Ramsey ergodic theory and in various other branches of mathematics.

\section{Partition theorems for $\omega$-located words} 

The purpose of this section is to prove a partition theorem for $\omega$-located words over an 
infinite countable alphabet $\Sigma=\{\alpha_1, \alpha_2, \ldots \}$, dominated by a sequence $\vec{k}=(k_n)_{n\in\nat}$ (Theorems~\ref{thm:block-Ramsey1}, ~\ref{thm:block-Ramsey2}, Corollary~\ref{central}) as well as a partition theorem for unlocated $\omega$-words over the alphabet $\Sigma$ dominated by the sequence $\vec{k}$ (Theorem~\ref{thm:block-Ramsey3}) providing proper extensions of Bergelson-Blass-Hindman's partition theorem (Theorem 4.1 in \cite{BBH}) for located words over a finite alphabet and of Carlson's partition theorem (Lemma 5.9 in \cite{C}) for unlocated words over a finite alphabet respectively. Consequences of this theory are strong simultaneous extentions of the partition theorems of Hindman \cite{H} and van der Waerden \cite{vdW} for general semigroups, including the partition theorems of Hindman and Strauss in \cite{HS} (Theorems 14.12, 14.15).

Let $\Sigma=\{\alpha_1, \alpha_2, \ldots \}$ be an infinite countable alphabet, ordered according to the natural numbers, and let $\vec{k}=(k_n)_{n\in\nat}\subseteq \nat$ be a sequence of natural numbers. An \textit{$\omega$-located word} over $\Sigma$ dominated by $\vec{k}$ is a function $w$ from a non-empty,  finite subset $F$ of $\nat$ into the alphabet $\Sigma$ such that $w(n)=w_n\in \{\alpha_1, \alpha_2, \ldots, \alpha_{k_n}\}$ for every $n\in F$. So, the set 
$L(\Sigma, \vec{k})$ of all (constant) $\omega$-located words over $\Sigma$ dominated by $\vec{k}$ is:   

$L(\Sigma, \vec{k})= \{w=w_{n_1}\ldots w_{n_l} : l\in\nat, n_1<\cdots<n_l\in \nat$ and $ w_{n_i}\in \{\alpha_1, \alpha_2, \ldots, \alpha_{k_{n_i}} \}$ 
\begin{center}
for every $1\leq i \leq l \}$,
\end{center}
where the set $dom(w)=\{n_1,\ldots ,n_l\}$ is the \textit{domain} of the $\omega$-located word $w=w_{n_1}\ldots w_{n_l}$.

Let $\upsilon \notin \Sigma$ be an entity which is called a \textit{variable}. The set $L(\Sigma, \vec{k} ; \upsilon) $ of all the \textit{variable  $\omega$-located word} over $\Sigma$ dominated by $\vec{k}$ is: 

$L(\Sigma, \vec{k} ; \upsilon) = \{w=w_{n_1}\ldots w_{n_l} : l\in\nat, n_1<\cdots<n_l\in \nat, w_{n_i}\in \{\upsilon, \alpha_1, \alpha_2, \ldots, \alpha_{k_{n_i}} \}$ 
\begin{center}
for every $1\leq i \leq l$ and there exists $1\leq i \leq l$ with $w_{n_i}=\upsilon \}$. 
\end{center}
We set $L(\Sigma\cup\{\upsilon\}, \vec{k})=L(\Sigma, \vec{k})\cup L(\Sigma, \vec{k} ; \upsilon)$.

We endow the set $L(\Sigma\cup\{\upsilon\}, \vec{k})$ with a relation defining for $w,u\in L(\Sigma\cup\{\upsilon\}, \vec{k})$
\begin{center}
$w < u \Leftrightarrow \max dom(w)< \min dom(u)$.
\end{center}
For $w=w_{n_1}\ldots w_{n_r},  u=u_{m_1}\ldots u_{m_l}\in L(\Sigma \cup\{\upsilon\}, \vec{k})$ with $w<u$ we define the concatenating word $w \star u = w_{n_1}\ldots w_{n_r}u_{m_1}\ldots u_{m_l}\in L(\Sigma \cup\{\upsilon\}, \vec{k})$. 

We will define now for every $p\in\nat\cup\{0\}$ the functions 
\begin{center}
$T_p : L(\Sigma\cup\{\upsilon\}, \vec{k}) \longrightarrow L(\Sigma\cup\{\upsilon\}, \vec{k})$. 
\end{center}
Let $w=w_{n_1}\ldots w_{n_l}\in L(\Sigma\cup\{\upsilon\}, \vec{k})$. We set $T_0(w)=w$ and for $p\in\nat$ we set 
$T_p(w)=u_{n_1}\ldots u_{n_l}$, where, for $1\leq i\leq l$, we define  $u_{n_i}=w_{n_i}$ if $w_{n_i}\in \Sigma$, $u_{n_i}=\alpha_{p} $ if $w_{n_i}=\upsilon$ and $k_{n_i}\geq p$ and finally $u_{n_i}=\alpha_{k_{n_i}} $ if $w_{n_i}=\upsilon$ and $k_{n_i}<p$.  
\newline
We remark that for every $p\in\nat\cup\{0\}$ and $w=w_{n_1}\ldots w_{n_l}\in L(\Sigma\cup\{\upsilon\}, \vec{k})$ we have $dom(T_p(w))=dom(w)$, $T_p(w)=w$ if $w\in L(\Sigma, \vec{k})$ and $T_p(w)=T_{p_w}(w)$ if $w\in L(\Sigma, \vec{k} ; \upsilon)$ and $p\geq p_w=k_{n_w}$, where $n_w=\max \{n_i : 1\leq i \leq l, w_{n_i}=\upsilon \}$. Also, $T_p(w\star u)=T_p(w)\star T_p(u)$ for $w,u\in L(\Sigma\cup\{\upsilon\}, \vec{k})$ with $w<u$ and $T_p( L(\Sigma\cup\{\upsilon\}, \vec{k}))=L(\Sigma, \vec{k})$ for every $p\in\nat$.

With the previous terminology we can state the partition theorem for $\omega$-located words.
\begin{thm}
[\textsf{Partition theorem for $\omega$-located words}]
\label{thm:block-Ramsey1}
Let $\Sigma=\{\alpha_1, \alpha_2, \ldots \}$ be an infinite countable alphabet, $\upsilon \notin \Sigma$ a variable, $\vec{k}=(k_n)_{n\in\nat}\subseteq \nat$ and $r,s\in\nat$. If $L(\Sigma, \vec{k} ; \upsilon)=A_1\cup\cdots\cup A_r$ and 
$L(\Sigma, \vec{k})=C_1\cup\cdots\cup C_s$, then there exist a sequence $(w_n)_{n\in\nat}\subseteq L(\Sigma, \vec{k} ; \upsilon)$ with $w_n<w_{n+1}$ for every $n\in\nat$ and $1\leq i_0 \leq r$, $1\leq j_0 \leq s$ such that
\begin{center}
$T_{p_1}(w_{n_1})\star \ldots \star T_{p_\lambda}(w_{n_\lambda})\in A_{i_0}$, 
\end{center}
for every $\lambda\in\nat$, $n_1<\cdots<n_\lambda\in\nat$, $p_1,\ldots,p_\lambda\in\nat\cup\{0\}$ such that $0\leq p_i \leq k_{n_i}$ for every $1\leq i \leq \lambda$ and $0 \in \{ p_1,\ldots,p_\lambda\}$; and 
\begin{center}
$T_{p_1}(w_{n_1})\star \ldots \star T_{p_\lambda}(w_{n_\lambda})\in C_{j_0}$, 
\end{center}
for every $\lambda\in\nat$, $n_1<\cdots<n_\lambda\in\nat$ and $p_1,\ldots,p_\lambda \in\nat$ such that $1\leq p_i \leq k_{n_i}$ for every $1\leq i \leq \lambda$.
\end{thm}

In the proof of Theorem~\ref{thm:block-Ramsey1} we will apply some results of the theory of left compact semigroups, which we mention below.  

\subsection*{Left compact semigroups}

A non-empty, \textit{left compact semigroup} is a semigroup $(X, +)$, $X\neq \emptyset$ endowed with a topology $\Tau$ such that $(X,\Tau)$ is a compact Hausdorff space and the maps $f_y : X\longrightarrow X$ with $f_y(x)=x+y$ for $x\in X$ are continuous for every $y\in X$. 

Let $(X, +)$ be a semigroup. An element $x$ of $X$ is called \textit{idempotent} of $(X, +)$ if $x+x=x$. According to a fundamental result due to Ellis, every non-empty, left compact semigroup contains an idempotent. On the set of all idempotents of $(X, +)$ is defined a partial order $\leq$ by the rule
\begin{center}
$x_1 \leq x_2 \Longleftrightarrow x_1+x_2=x_2+x_1=x_1$.
\end{center}
An idempotent $x$ of $(X, +)$ is called \textit{minimal for $X$} if every idempotent $x_1$ of $X$ satisfing the relation $x_1\leq x $ is equal to $x$. According to \cite{FuK}, for every idempotent $x$ of a non-empty, left compact semigroup $(X, +)$ there exists an idempotent $x_1$ of $X$ which is minimal for $X$ and $x_1\leq x$. Also, every two-sided ideal of $X$ contains all the minimal for $X$ idempotents of $X$ (a subset $I$ of $X$ is called \textit{two-sided ideal} of $(X,+)$ if $X+I\subseteq I$ and $I+X\subseteq I$).

\subsection*{Ultrafilters}

Let $X$ be a non-empty set. An \textit{ultrafilter} on the set $X$ is a zero-one finite additive measure $\mu$ defined on all subsets of $X$. The set of all ultrafilters on the set $X$ is denoted by $\beta X$. So, $\mu\in\beta X$ if and only if 
\begin{itemize}
\item[{(i)}] $\mu(A)\in\{0,1\}$ for every $A\subseteq X$, $\mu(X)=1$, and 
\item[{(ii)}] $\mu(A\cup B)=\mu(A)+\mu(B)$ for every $A,B\subseteq X$ with $A\cap B=\emptyset$.
\end{itemize}
For every $x\in X$ is defined the ultrafilter $\mu_x$ on $X$ corresponding a set $A\subseteq X$ to $\mu_x(A)=1$ if $x\in A$ and  $\mu_x(A)=0$ if $x\notin A$. The ultrafilters $\mu_x$ for $x\in X$ are called \textit{principal ultrafilters} on $X$. So, $\mu$ is a non-principal ultrafilter on $X$ if and only if  $\mu(A)=0$ for every finite subset $A$ of $X$. It is easy to see that for $\mu\in\beta X$ and $A\subseteq X$ with $\mu(A)=1$ we have $\mu(X\setminus A)=0$, $\mu(B)=1$ for 
every $B\subseteq X$ with $A\subseteq B$ and $\mu(A\cap B)=1$ for every $B\subseteq X$ with $\mu(B)=1$.

The set $\beta X$ becomes a compact Hausdorff space if it be endowed with the topology $\Tau$ which has basis the family $\{A^* :  A\subseteq X \}$, where $A^*=\{\mu\in\beta X : \mu(A)=1 \}$. It is easy to see that $(A\cap B)^*=A^*\cap B^*$, $(A\cup B)^*=A^*\cup B^*$ and $(X\setminus A)^*=\beta X\setminus A^*$ for every $A,B\subseteq X$. We always consider the set $\beta X$ endowed with the topology $\Tau$. 

Let a function $T : X\longrightarrow Y$. Then the function 
\begin{center}
$\beta T : \beta X\longrightarrow \beta Y$  with  $\beta T(\mu)(B)=\mu(T^{-1}(B))$ for $\mu\in\beta X$ and $B\subseteq Y$  
\end{center}
is always continuous.

If $(X,+)$ is a semigroup, then a binary operation $+$ is defined on $\beta X$ corresponding to every $\mu_1, \mu_2\in \beta X$ the ultrafilter $\mu_1 + \mu_2\in \beta X$ given by
\begin{center}
$(\mu_1 + \mu_2)(A)= \mu_1(\{x\in X : \mu_2(\{y\in X : x+y\in A\})=1\})$ for every $A\subseteq X$.
\end{center}
With this operation the set $\beta X$ becomes a semigroup and for every $\mu\in\beta X$ the function $f_\mu : \beta X\longrightarrow \beta X$ with $f_\mu(\mu_1)=\mu_1+\mu$ is continuous.

Hence, if $(X,+)$ is a semigroup, then $(\beta X, +)$ becomes a left compact semigroup.

\begin{proof}[Proof of Theorem~\ref{thm:block-Ramsey1}]
Let $X=L(\Sigma\cup\{\upsilon\}, \vec{k})$, $Y=L(\Sigma, \vec{k})$ and $Z=L(\Sigma, \vec{k} ; \upsilon)$.  
We endow the set $X$ with an operation $+$ defining for $w=w_{n_1}\ldots w_{n_r}, u=u_{m_1}\ldots u_{m_l}\in X $  
\begin{center}
$w+u=z_{q_1}\ldots z_{q_s}\in X$,
\end{center}
where $\{q_1,\ldots,q_s\}=\{n_1,\ldots,n_r\}\cup \{m_1,\ldots,m_l\}$ with $q_1<\ldots<q_s$ and, for $1\leq i\leq s$, $z_{q_i}=w_{q_i}$ if $q_i\notin \{m_1,\ldots,m_l\}$, $z_{q_i}=u_{q_i}$ if $q_i\notin \{n_1,\ldots,n_r\}$ and if $q_i\in\{n_1,\ldots,n_r\}\cap \{m_1,\ldots,m_l\}$, 
then $z_{q_i}=\alpha_{\max\{\mu, \nu\}}$ in case $w_{q_i}=\alpha_\mu$ and $u_{q_i}=\alpha_\nu$, and 
$z_{q_i}=\upsilon$ in case either $w_{q_i}=\upsilon$ or $w_{q_i}=\upsilon$.
\newline
Observe that $(X, +)$ is a semigroup and $(Y, +)$, $(Z, +)$ are subsemigroups of $(X,+)$. Also, $w+u=w\star u$ if $w, u\in X$ and $w<u$.

Since $(X, +)$ is a semigroup, $(\beta X, +)$ has the structure of a left compact semigroup as described above. For every $A\subseteq X$ and $w\in X$ we set
\begin{center}
$A_w=\{u\in A : w<u\}$ and
\end{center}
\begin{center}
$\theta A=\bigcap \{(A_w)^* : w\in X \}$,
\end{center}
where $(A_w)^*=\{\mu\in\beta X : \mu(A_w)=1 \}$. 

\noindent {\bf Claim 1} 
Let $A$ be a non-empty subset of $X$ satisfing
\begin{itemize}
\item[{(i)}] $w\star u\in A$ for every $w,u\in A$ with $w<u$ and
\item[{(ii)}] for every $n\in\nat$ there exists $u\in A$ with $n<\min dom(u)$.
\end{itemize}
Then $\theta A\subseteq A^*$ is a non-empty left compact subsemigroup of $\beta X$ and all the elements of $\theta A$ are non-principal ultrafilters on $X$.

Indeed, for every $w\in X$ the set $(A_w)^*=\beta X\setminus(X\setminus A_w)^*$ is a compact subset of $\beta X$ , so $\theta A$ is a compact subset of $\beta X$. The set $A$ satisfies property (ii), so for every $w\in X$, we have $A_w\neq\emptyset$ and consequently $(A_w)^*\neq\emptyset$, since $\mu_u\in (A_w)^*$ if $u\in A_w$. Moreover, according to property (ii), the family $\{(A_w)^* : w\in X\}$ has the finite intersection property, hence $\theta A\neq\emptyset$. Since $A$ satisfies property (i), 
$(\theta A, +)$ is a semigroup. Indeed, for $\mu_1,\mu_2\in \theta A$ and $w\in X$, 
\newline
$\mu_1 +\mu_2(A_w)=\mu_1(\{u_1\in A_w : \mu_2(\{u_2\in A_{u_1} : u_1+u_2\in A_w\})=1 \})=$ 
\newline
$=\mu_1(\{u_1\in A_w : \mu_2( A_{u_1})=1 \})=\mu_1(A_w)=1$. 
\newline
Hence, $\theta A$ is a non-empty left compact subsemigroup of $\beta X$. Since $w\notin A_w$, we have that $\mu_w\notin \theta A$ for every $w\in X$, so every $\mu\in \theta A$ is a non-principal ultrafilter on $X$.

According to the claim, $ \theta X, \theta Y$ and $\theta Z$ are non-empty left compact subsemigroups of $\beta X$ consisted of non-principal ultrafilters on $X$. Notice that $\theta Y \subseteq \theta X $ and that $\theta Z\subseteq \theta X$. Moreover, $\theta Z$ is a two sided ideal of $\theta X$. Indeed, for $\mu_1\in \theta Z$, $\mu_2\in \theta X$ and $w\in X$,
\newline
$\mu_1+\mu_2(Z_w)=\mu_1(\{u_1\in Z_w : \mu_2(\{u_2\in X_{u_1} : u_1+u_2\in Z_w\})=1 \})=$
\newline
$=\mu_1(\{u_1\in Z_w : \mu_2(X_{u_1})=1 \})=\mu_1(Z_w)=1=\mu_2+\mu_1(Z_w)$. 

Let $p\in\nat$ and let the continuous function $\beta T_p : \beta X\longrightarrow \beta X$ with 
$\beta T_p(\mu)(A)=\mu((T_p)^{-1}(A))$ for every $\mu\in\beta X$ and $A\subseteq X$. We note that:

(i) $\beta T_p(\theta X)\subseteq \theta Y\subseteq \theta X$, since, for $\mu\in \theta X$ and $w\in X$, 
\newline
$\beta T_p(\mu)(Y_w)=\mu(\{u\in X_w : T_p(u)\in Y_w \})=\mu(X_w)=1$,

(ii) $\beta T_p(\mu)=\mu$ for every $\mu\in\theta Y$, since, for $\mu\in \theta Y$ and $A\subseteq X$,
\newline 
$\beta T_p(\mu)(A)=\mu(\{u\in Y : T_p(u)= u\in A\})=\mu(A\cap Y)=\mu(A)$, and

(iii) for $\mu_1, \mu_2\in \theta X$ and $A\subseteq X$, we have 
\newline
$\beta T_p(\mu_1+\mu_2)(A)= \mu_1(\{u_1\in X : \mu_2(\{u_2\in X_{u_1} : T_p(u_1+u_2)\in A\})=1\})=$
\newline
$=\mu_1(\{u_1\in X : \mu_2(\{u_2\in X_{u_1} : T_p(u_1)+T_p(u_2)\in A\})=1\})=$
\newline
$=\beta T_p(\mu_1)+ \beta T_p(\mu_2)(A)$.

Let $\mu_1$ be an idempotent in in the non-empty, left compact semigroup $\theta Y$ minimal for $\theta Y$. Since $\theta Z$ is a two sided ideal of the left compact semigroup $\theta X$ and 
$\mu_1\in \theta Y\subseteq\theta X$ is an idempotent of $\theta X$, there exists an idempotent $\mu\in \theta Z\subseteq \beta X$ minimal for $\theta X$ with $\mu\leq \mu_1$. Since for each $p\in\nat$ the restriction of $\beta T_p$ to $\theta X$ is an homomorphism, we have that $\beta T_p(\mu)\leq \beta T_p(\mu_1)$ for every $p\in\nat$. But $\beta T_p(\mu_1)= \mu_1$ for every $p\in\nat$, since $\mu_1\in\theta Y$, hence, $\beta T_p(\mu)\leq \mu_1$ for every $p\in\nat$. Now, since $\mu_1$ is minimal for $\theta Y$ and 
$\beta T_p(\mu)\in \theta Y$ for every $p\in\nat$, we have $\beta T_p(\mu)=\mu_1$ for every $p\in\nat$.

In conclution, for every idempotent $\mu_1\in\theta Y\subseteq \theta X\subseteq \beta X$ minimal for $\theta Y$ there exists an idempotent $\mu\in\theta Z \subseteq\theta X\subseteq \beta X$ minimal for $\theta X$ such that: $\mu+\mu=\mu$, $\mu_1+\mu_1=\mu_1 $, $\mu_1=\beta T_p(\mu)$ for every $p\in\nat$, 
and $\mu+\mu_1=\mu_1+\mu=\mu $.

We will construct, by induction on $n$, the required sequence $(w_n)_{n\in\nat}\subseteq Z$. Since, $Z=A_1\cup\cdots\cup A_r$ and $Y=C_1\cup\cdots\cup C_s$, there exist $1\leq i_0 \leq r$, $1\leq j_0 \leq s$ such that $\mu(A_{i_0})=1$ and $\mu_1(C_{j_0})=1$. Let $w\in Z$. Since $ \mu_1=\beta T_p(\mu)=\beta T_p(\mu)+\mu_1$ and $\mu=\mu+\mu=\beta T_p(\mu)+\mu=\mu+\mu_1$ for every 
$p\in\nat$ with $p\leq k$, starting with $w\in Z$, $B_1=A_{i_0}$ and $D_1=C_{j_0}$, can be constructed an increasing sequence $w<w_1<w_2<\cdots$ in $Z$ and two decreasing sequences $Z\supseteq B_1\supseteq B_2\supseteq \cdots$, and $Y\supseteq D_1\supseteq D_2\supseteq \cdots$ such that for every $n\in\nat$ to satisfy:
\newline
$\mu(B_n)=1$ and $\mu_1(D_n)=1$,
\newline
$w_n\in B_n$ and $T_p(w_n)\in D_n$ for every $p\in\nat$ with $p\leq k_n$,  
\newline
$B_{n+1}=\{u\in (B_n)_{w_n} : w_n + u\in B_n$ and $ T_{p}(w_n)+ u\in B_n$ for all $ p\in\nat$ with $p\leq k_n \}$, and
\newline
$D_{n+1}=\{z\in (D_n)_{w_n} : w_n+z\in B_n$ and $T_{p}(w_n)+ z\in D_n$ for all $ p\in\nat$ with $p\leq k_n \}$.

We claim that the sequence $(w_n)_{n\in\nat}$ has the required properties. We will prove, by induction on $\lambda$, that, for every $\lambda\in\nat$, 

$T_{p_1}(w_{n_1})\star \ldots \star T_{p_\lambda}(w_{n_\lambda})\in B_{n_1}\subseteq B_{1}=A_{i_0}$, 
\newline
for every $n_1<\cdots<n_\lambda\in\nat$, $p_1,\ldots,p_\lambda\in\nat\cup\{0\}$ such that $0\leq p_i \leq k_{n_i}$ for every $1\leq i \leq \lambda$ and $0 \in \{ p_1,\ldots,p_\lambda\}$, and also

$T_{p_1}(w_{n_1})\star \ldots \star T_{p_\lambda}(w_{n_\lambda})\in D_{n_1}\subseteq D_{1}=C_{j_0}$, 
\newline
for every $n_1<\cdots<n_\lambda\in\nat$, $p_1,\ldots,p_\lambda \in\nat$ such that $1\leq p_i \leq k_{n_i}$ for every $1\leq i \leq \lambda$. 

Indeed, for $n_1\in\nat$ and $p_1\in\nat$ such that $1\leq p_1 \leq k_{n_1}$, we have $w_{n_1}\in B_{n_1}$ and $T_{p_1}(w_{n_1})\in D_{n_1}$.  Assume that the accertion holds for $\lambda\geq 1$ and let $n_1<\cdots<n_\lambda<n_{\lambda+1}\in\nat$ and $ p_1,\ldots,p_\lambda, p_{\lambda+1}\in\nat\cup\{0\}$ such that $0\leq p_i \leq k_{n_i}$ for every $1\leq i \leq \lambda+1$. 

\noindent {\bf Case 1} Let $0\in \{p_2,\ldots,p_\lambda, p_{\lambda+1}\}$. Then $u=T_{p_2}(w_{n_2})\star \ldots \star T_{p_{\lambda+1}}(w_{n_{\lambda+1}})\in B_{n_2}\subseteq B_{n_1 +1}$, according to the induction hypothesis. Hence, 
\newline
$ T_{p_1}(w_{n_1})+ u=T_{p_1}(w_{n_1})\star \ldots \star T_{p_{\lambda+1}}(w_{n_{\lambda+1}})\in B_{n_1}$.

\noindent {\bf Case 2} Let $0\notin \{p_2,\ldots,p_\lambda, p_{\lambda+1}\}$. Then  $z=T_{p_2}(w_{n_2})\star \ldots \star T_{p_{\lambda+1}}(w_{n_{\lambda+1}})\in D_{n_2}\subseteq D_{n_1 +1}$, according to the induction hypothesis. If $p_1=0$, then 
\newline
$ w_{n_1}+ z=T_{p_1}(w_{n_1})\star \ldots \star T_{p_{\lambda+1}}(w_{n_{\lambda+1}})\in B_{n_1}$. 
\newline
If $p_1\in\nat$, then 
\newline
$ T_{p_1}(w_{n_1})+ z=T_{p_1}(w_{n_1})\star \ldots \star T_{p_{\lambda+1}}(w_{n_{\lambda+1}})\in D_{n_1}$. 
\newline
This finishes the proof.
\end{proof}

\begin{remark}\label{rem. C}
The special case of Theorem~\ref{thm:block-Ramsey1} for sequences  $\vec{k}=(k_n)_{n\in\nat}\subseteq \nat$ with  $k_n=k_1$ for every $n\in\nat$ coincide with the partition theorem of Bergelson, Blass and Hindman (Theorem 4.1 in  \cite{BBH}) for located words over a finite alphabet, while the case $\vec{k}=(k_n)_{n\in\nat}$ with  $k_n=1$ for every $n\in\nat$ gives Hindman's partition theorem (\cite{H}, \cite{Ba}).
\end{remark}
\begin{cor}
\label{central}
Let $\Sigma=\{\alpha_1, \alpha_2, \ldots \}$ be an infinite countable alphabet, $\upsilon \notin \Sigma$ a variable and $\vec{k}=(k_n)_{n\in\nat}\subseteq \nat$. Then for every 
$A\subseteq L(\Sigma, \vec{k})$ with $\mu(A)=1$ for $\mu\in\theta L(\Sigma, \vec{k})$ minimal in $\theta L(\Sigma, \vec{k})$ there exists a sequence $(w_n)_{n\in\nat}\subseteq L(\Sigma, \vec{k} ; \upsilon)$ with $w_n<w_{n+1}$ for every $n\in\nat$ such that
\begin{center}
$T_{p_1}(w_{n_1})\star \ldots \star T_{p_\lambda}(w_{n_\lambda})\in A$, 
\end{center}
for every $\lambda\in\nat$, $n_1<\cdots<n_\lambda\in\nat$ and $p_1,\ldots,p_\lambda \in\nat$ with $1\leq p_i \leq k_{n_i}$ for $1\leq i \leq \lambda$.
\end{cor}

Now we will prove a stroger version of Theorem~\ref{thm:block-Ramsey1}, using the notion of extracted $\omega$-located words of a given orderly sequence of variable $\omega$-located words defined below. 

\subsection*{Extracted $\omega$-located words, Extractions}

Let $\Sigma=\{\alpha_1, \alpha_2, \ldots \}$ be an infinite countable alphabet ordered according to the natural numbers, $\upsilon \notin \Sigma$ a variable and $\vec{k}=(k_n)_{n\in\nat}\subseteq \nat$ an increasing sequence. We set

$L^\infty (\Sigma, \vec{k} ; \upsilon) = \{\vec{w} = (w_n)_{n\in\nat} : w_n\in L(\Sigma, \vec{k} ; \upsilon)$
and  $w_n<w_{n+1} \ \forall\ n\in\nat\}$. 
\newline
Let a sequence $\vec{w} = (w_n)_{n\in\nat}\in L^\infty (\Sigma, \vec{k} ; \upsilon)$.
  
An \textit{extracted variable $\omega$-located word} of $\vec{w}$ is a variable $\omega$-located word 
\begin{center}
$u=T_{p_1}(w_{n_1})\star \ldots \star T_{p_\lambda}(w_{n_\lambda})\in L(\Sigma, \vec{k} ; \upsilon)$, 
\end{center}
where $\lambda\in\nat$, $n_1<\cdots<n_\lambda\in\nat$, $p_1,\ldots,p_\lambda\in\nat\cup\{0\}$ with $0\leq p_i \leq k_{n_i}$ for every $1\leq i \leq \lambda$ and $0 \in \{ p_1,\ldots,p_\lambda\}$. 
\newline
The set of all the extracted variable $\omega$-located words of $\vec{w}$ is denoted by $EV(\vec{w})$. 

An \textit{extracted $\omega$-located word} of $\vec{w}$ is an $\omega$-located word 
\begin{center}
$z=T_{p_1}(w_{n_1})\star \ldots \star T_{p_\lambda}(w_{n_\lambda})\in L(\Sigma, \vec{k})$, 
\end{center}
where $\lambda\in\nat$, $n_1<\cdots<n_\lambda\in\nat$, $p_1,\ldots,p_\lambda \in\nat$ with $1\leq p_i \leq k_{n_i}$ for every $1\leq i \leq \lambda$. 
\newline
The set of all the extracted $\omega$-located words of $\vec{w}$ is denoted by $E(\vec{w})$. Let

$EV^{\infty}(\vec{w}) = \{\vec{u}=(u_n)_{n\in\nat} \in L^\infty (\Sigma, \vec{k} ; \upsilon) : u_n\in EV(\vec{w})$ for every $n\in\nat \}$.
\newline
If $\vec{u} \in EV^{\infty}(\vec{w})$, then we say that $\vec{u}$ is an \textit{extraction} of $\vec{w}$ and we write $\vec{u} \prec \vec{w}$. Notice that $\vec{u} \prec \vec{w}$ if and only if $EV(\vec{u})\subseteq EV(\vec{w})$ and that $\vec{w} \prec \vec{e}$ for every $\vec{w}\in L^{\infty} (\Sigma, \vec{k} ; \upsilon)$ where $\vec{e}=(e_n)_{n\in\nat}\in L^{\infty} (\Sigma, \vec{k} ; \upsilon)$ with $e_n : \{n\}\rightarrow \{\upsilon\}$ for every $n\in \nat$.

\begin{thm}
\label{thm:block-Ramsey2}
Let $\Sigma=\{\alpha_1, \alpha_2, \ldots \}$ be an infinite countable alphabet, $\upsilon \notin \Sigma$ a variable, $\vec{k}=(k_n)_{n\in\nat}\subseteq \nat$ an increasing sequence, $\vec{w} = (w_n)_{n\in\nat}\in L^\infty (\Sigma, \vec{k} ; \upsilon)$ and $r,s\in\nat$. If $L(\Sigma, \vec{k} ; \upsilon)=A_1\cup\cdots\cup A_r$ and 
$L(\Sigma, \vec{k})=C_1\cup\cdots\cup C_s$, then there exist an extraction $\vec{u}=(u_n)_{n\in\nat}$ of $\vec{w}$ and $1\leq i_0 \leq r$, $1\leq j_0 \leq s$ such that 
\begin{center}
$EV(\vec{u})\subseteq A_{i_0}$ and $E(\vec{u})\subseteq C_{j_0}$.
\end{center}
\end{thm}
\begin{proof}
Let the function $\phi:L(\Sigma\cup\{\upsilon\}, \vec{k})\rightarrow EV(\vec{w})\cup E(\vec{w})$ which sends $t=t_{n_1}\ldots t_{n_\lambda}\in L(\Sigma\cup\{\upsilon\}, \vec{k})$ to 
$\phi(t_{n_1}\ldots t_{n_\lambda})=T_{p_1}(w_{n_1})\star \ldots \star T_{p_\lambda}(w_{n_\lambda})$, where for $1\leq i \leq \lambda$, $p_i=0$ if $t_{n_i}=\upsilon$ and $p_i=\mu$ if $t_{n_i}=\alpha_{\mu}$ for $1\leq \mu \leq k_{n_i}$. The function $\phi$ is one to one and onto $EV(\vec{w})\cup E(\vec{w})$. Moreover $\phi(L(\Sigma, \vec{k} ; \upsilon))=EV(\vec{w})$ and $\phi(L(\Sigma, \vec{k}))=E(\vec{w})$.

According to Theorem~\ref{thm:block-Ramsey1}, there exist a sequence $\vec{s}=(s_n)_{n\in\nat}\in L^\infty (\Sigma, \vec{k}; \upsilon)$ and $1\leq i_0 \leq r$, $1\leq j_0 \leq s$ such that $EV(\vec{s})\subseteq (\phi)^{-1}(A_{i_0})$ and $E(\vec{s})\subseteq (\phi)^{-1}(C_{j_0})$. Set $u_n=\phi(s_n)\in EV(\vec{w})$ for every $n\in\nat$ and $\vec{u}=(u_n)_{n\in\nat}$. Then $\vec{u}=(u_n)_{n\in\nat}$ is an extraction of $\vec{w}$ such that $EV(\vec{u})\subseteq \phi(EV(\vec{s}))\subseteq A_{i_0}$ and $E(\vec{u})\subseteq \phi(E(\vec{s}))\subseteq C_{j_0}$. 
\end{proof}
As a consequence of Theorem~\ref{thm:block-Ramsey1}, we will prove a partition theorem for (unlocated) $\omega$-words, extending Carlson's partition theorem (Lemma 5.9 in \cite{C}) for words over a finite alphabet. 
\newline
For an infinite countable alphabet $\Sigma=\{\alpha_1, \alpha_2, \ldots \}$, $\upsilon \notin \Sigma$ and $\vec{k}=(k_n)_{n\in\nat}\subseteq \nat$, let $W(\Sigma, \vec{k})$ be the set of $\omega$-located words over $\Sigma$ dominated by $\vec{k}$ with domain an initial segment of $\nat$ and $W(\Sigma, \vec{k}; \upsilon)$ the set of variable $\omega$-located words over $\Sigma$ dominated by $\vec{k}$ with domain an initial segment of $\nat$. 
\newline
We set $W(\Sigma\cup\{\upsilon\}, \vec{k})=W(\Sigma, \vec{k})\cup W(\Sigma, \vec{k} ; \upsilon)$ and 
we endow the set $W(\Sigma\cup\{\upsilon\}, \vec{k})$ with a relation defining for $w,u\in W(\Sigma\cup\{\upsilon\}, \vec{k})$
\begin{center}
$w=w_{1}\ldots w_{l} < u=u_{1}\ldots u_{r} \Leftrightarrow l<r$ and $u_{1}=\ldots=u_{l}=\alpha_1$.
\end{center}
For $w=w_{1}\ldots w_{l}, u=u_{1}\ldots u_{r}\in W(\Sigma \cup\{\upsilon\}, \vec{k})$ with $w<u$ we define the concatenating $\omega$-word $w \star u = w_{1}\ldots w_{l}u_{l+1}\ldots u_{r}\in W(\Sigma \cup\{\upsilon\}, \vec{k})$. 

Let $W^\infty (\Sigma, \vec{k}; \upsilon) = \{\vec{w} = (w_n)_{n\in\nat} : w_n\in W(\Sigma, \vec{k} ; \upsilon)$
and  $w_n<w_{n+1} \ \forall\ n\in\nat\}$.

\begin{thm}
[\textsf{Partition theorem for $\omega$-words}]
\label{thm:block-Ramsey3}
Let $\Sigma=\{\alpha_1, \alpha_2, \ldots \}$ be an infinite countable alphabet, $\upsilon \notin \Sigma$ a variable, $\vec{k}=(k_n)_{n\in\nat}\subseteq \nat$ an increasing sequence and $r,s\in\nat$. If $W(\Sigma, \vec{k} ; \upsilon)=A_1\cup\cdots\cup A_r$ and $W(\Sigma, \vec{k})=C_1\cup\cdots\cup C_s$, there exist a sequence $(u_n)_{n\in\nat}\in W^\infty (\Sigma, \vec{k}; \upsilon)$ and $1\leq i_0 \leq r$, $1\leq j_0 \leq s$ such that
\begin{center}
$T_{p_1}(u_{n_1})\star \ldots \star T_{p_\lambda}(u_{n_\lambda})\in A_{i_0}$ 
\end{center}
for every $\lambda\in\nat$, $n_1<\cdots<n_\lambda\in\nat$, $p_1,\ldots,p_\lambda\in\nat\cup\{0\}$ such that $0\leq p_i \leq k_{n_i}$ for every $1\leq i \leq \lambda$ and $0 \in \{ p_1,\ldots,p_\lambda\}$; and 
\begin{center}
$T_{p_1}(u_{n_1})\star \ldots \star T_{p_\lambda}(u_{n_\lambda})\in C_{j_0}$, 
\end{center}
for every $\lambda\in\nat$, $n_1<\cdots<n_\lambda\in\nat$ and $p_1,\ldots,p_\lambda \in\nat$ such that $1\leq p_i \leq k_{n_i}$ for every $1\leq i \leq \lambda$.
\end{thm}
\begin{proof}
Let the function 
$f : L(\Sigma\cup\{\upsilon\}, \vec{k}) \longrightarrow W(\Sigma\cup\{\upsilon\}, \vec{k})$ which sends 
\newline
$w=w_{n_1}\ldots w_{n_l}\in L(\Sigma\cup\{\upsilon\}, \vec{k})$ to $f(w)=u=u_1\ldots u_{n_l}\in W(\Sigma\cup\{\upsilon\}, \vec{k})$ where $u_{n_j}=w_{n_j}$ for every $1\leq j\leq l$ and $u_i=\alpha_1$ for every $i\in\{1,\ldots, n_l\}\setminus \{n_j : 1\leq j\leq l \}$. Then $f(L(\Sigma, \vec{k} ; \upsilon))\subseteq W(\Sigma, \vec{k} ; \upsilon)$ and $f(L(\Sigma, \vec{k}))\subseteq W(\Sigma, \vec{k})$. Also, for $w_1,w_2\in L(\Sigma\cup\{\upsilon\}, \vec{k})$ with $w_1<w_2$ we have $f(w_1)<f(w_2)$, $f(w_1 \star w_2)=f(w_1) \star f(w_2)$ and $f(T_p(w_1))=T_p(f(w_1))$ for every $p\in\nat$.

According to Theorem~\ref{thm:block-Ramsey1}, 
there exist a sequence $(w_n)_{n\in\nat}\in L^\infty (\Sigma, \vec{k} ; \upsilon)$ and 
$1\leq i_0 \leq r$, $1\leq j_0 \leq s$ such that: $T_{p_1}(w_{n_1})\star \ldots \star T_{p_\lambda}(w_{n_\lambda})\in f^{-1}(A_{i_0})$, 
for every $\lambda\in\nat$, $n_1<\cdots<n_\lambda\in\nat$, $p_1,\ldots,p_\lambda\in\nat\cup\{0\}$ such that $0\leq p_i \leq k_{n_i}$ for every $1\leq i \leq \lambda$ and $0 \in \{ p_1,\ldots,p_\lambda\}$; and 
$T_{p_1}(w_{n_1})\star \ldots \star T_{p_\lambda}(w_{n_\lambda})\in f^{-1}(C_{j_0})$, 
for every $\lambda\in\nat$, $n_1<\cdots<n_\lambda\in\nat$ and $p_1,\ldots,p_\lambda \in\nat$ such that $1\leq p_i \leq k_{n_i}$ for every $1\leq i \leq \lambda$.

Set $u_n=f(w_n)$ for every $n\in\nat$. The sequence $(u_n)_{n\in\nat}\in W^\infty (\Sigma, \vec{k}; \upsilon)$ satisfies the required properties.
\end{proof}

We will define now a function $g$ from the set $L(\nat, \vec{k})$ of all the $\omega$-located words over $\nat$ dominated by an increasing sequence $\vec{k}$ into an arbitrary semigroup. Via the function $g$ can be proved strong partition theorems for (commutative or noncommutative) semigroups as consequences of  Theorems~\ref{thm:block-Ramsey1}, ~\ref{thm:block-Ramsey2}, ~\ref{thm:block-Ramsey3}. 
\newline
Let $(X,+)$ be a semigroup, $\vec{k}=(k_n)_{n \in \nat}\subseteq \nat$ an increasing sequence, $(y_{l,n})_{n \in \nat}\subseteq X$ for every $l \in \nat,$ and the alphabet $\Sigma=\{1,2,\ldots\}=\nat$. We define the function 
\begin{center}
$g:L(\Sigma,\vec{k})\rightarrow X$ with $g(w_{n_1}\ldots w_{n_l})=\sum^{l}_{i=1}y_{w_{n_i},\;n_i}.$ 
\end{center}
Observe that $g(u_1\star u_2)=g(u_1)+g(u_2)$ for every $u_1<u_2\in L(\Sigma,\vec{k})$. 

Let $w=w_{n_1}\ldots w_{n_l}\in L(\Sigma, \vec{k} ; \upsilon)$ with $w_{n_1}, w_{n_l}\in \Sigma$. If $(X,+)$ is a commutative semigroup, then for every $p\in\nat$ with $1\leq p\leq k_{n_1}$
\begin{center}
$g(T_{p}(w))=\sum_{t\in E\setminus H}y_{w_t,t}+\sum_{t \in H}y_{p, t}$ 
\end{center}
where, $E=\{n_i:\; 1\leq i\leq l\}$ is the domain of $w$ and $H=\{n \in E: w_n=\upsilon \}$. If $(X,+)$ is a noncommutative semigroup, then there exist $m\in\nat$ such that $E=E_1\cup\ldots \cup E_{m+1}$ with 
$\max E_i<\min E_{i+1}$ for $1\leq i\leq m$ and $H=H_1\cup\ldots \cup H_m$ with $\emptyset\neq H_i \subseteq E_i$ for every $1\leq i\leq m$ such that for every $p\in\nat$ with $1\leq p\leq k_{n_1}$
\begin{center}
$g(T_{p}(w))=(\sum^{m}_{i=1}(\sum_{t\in E_i\setminus H_i}y_{w_t,t}+\sum_{t \in H_{i}}y_{p,\;t}))+\sum_{t\in E_{m+1}}y_{w_t,t}$.
\end{center}
Via the function $g$, Theorem~\ref{thm:block-Ramsey1} implies the following partition theorem for commutative semigroups, which is an improvement of Theorem 14.12 of Hindman and Strauss in \cite{HS}, and also an analogous theorem for noncommutative semigroups (Theorem 14.15 in \cite{HS}). 
For a sequence $(x_n)_{n\in\nat}$ in a semigroup $(X,+)$, we set

$FS\big((x_n)_{n\in\nat}\big)=\{x_{n_1}+\ldots+x_{n_\lambda} : \lambda\in\nat, n_1<\cdots<n_\lambda\in\nat\}$.
\begin{cor}
\label{thm:van der Waerden}
Let $(X,+)$ be a commutative semigroup, $\vec{k}=(k_n)_{n \in \nat}\subseteq \nat$ an increasing sequence and $(y_{l,n})_{n \in \nat}\subseteq X$ for every $l \in \nat,$. If $X=A_1\cup \ldots \cup A_r,\; r \in \nat$ is a finite partition of $X$, then there exist $1\leq i_0\leq r$, sequences $(E_n)_{n \in \nat}$, $(H_n)_{n \in \nat},$ of non-empty finite subsets of $\nat$ with $\max E_n<\min E_{n+1}$, $H_n\subseteq E_{n}$ for every $n \in \nat$ and a sequence $(\beta_n)_{n \in \nat}\subseteq X$ with $\beta_n=\sum_{j\in E_n \setminus H_n}y_{l^n_j,j}$ for $1\leq l^n_j \leq k_j$ such that for every function $f: \nat\rightarrow \nat$ with $f(n)\leq k_n$ for every $n\in\nat$
\begin{center}
$ FS\big((\beta_{n}+\sum_{t \in H_{n}}y_{f(n), t})_{n\in\nat}\big)\subseteq A_{i_0}$.
\end{center}
\end{cor}
\begin{proof}
Since  $L(\Sigma,\vec{k})=g^{-1}(A_1)\cup\ldots \cup g^{-1}(A_r)$, according to Theorem~\ref{thm:block-Ramsey1}, there exist $(w_n)_{n\in\nat}\subseteq L(\Sigma, \vec{k} ; \upsilon)$ with $w_n<w_{n+1}$ for every $n\in\nat$ and $1\leq i_0 \leq r$ satisfing
\begin{center}
$T_{p_1}(w_{n_1})\star \ldots \star T_{p_\lambda}(w_{n_\lambda})\in g^{-1}(A_{i_0})$, 
\end{center}
for every $\lambda\in\nat$, $n_1<\cdots<n_\lambda\in\nat$ and $p_1,\ldots,p_\lambda \in\nat$ with $1\leq p_j \leq k_{n_j}$ for $1\leq j \leq \lambda$.

Let $w_n=w^{n}_{q^{n}_1}\ldots w^{n}_{q^{n}_{l_n}}$ for every $n \in \nat$. We can suppose that $w^{n}_{q^{n}_1}, w^{n}_{q^{n}_{l_n}}\in \Sigma$ for every $n \in \nat$. Otherwise replace the sequence $(w_n)_{n\in\nat}$ by the sequence $(u_n)_{n\in\nat}$, where $u_n=T_1(w_{3n-1})\star w_{3n}\star T_1(w_{3n+1})$. We set 
$E_n =\{q^{n}_i:\; 1\leq i\leq l_n\}$ the domain of $w_n$, 
$H_n=\{q^{n}_i \in E_n: w^{n}_{q^{n}_{i}}=\upsilon \}$ and 
$\beta_n=\sum_{j\in E_n\setminus H_n}y_{w^n_j,j}\in X$. Then for every $\lambda\in\nat$, $n_1<\cdots<n_\lambda\in\nat$ and $p_1,\ldots,p_\lambda \in\nat$ with $1\leq p_j \leq k_{n_j}$ for $1\leq j \leq \lambda$

$(\beta_{n_1}+\sum_{t \in H_{n_1}}y_{p_1, t})+\ldots+(\beta_{n_\lambda}+
\sum_{t \in H_{n_\lambda}}y_{p_\lambda,t})=g(T_{p_1}(w_{n_1})\star \ldots \star T_{p_\lambda}(w_{n_\lambda}))\in A_{i_0}$.
\end{proof}

\section{Extended Ramsey type partition theorems for $\omega$-located words}

We will state and prove, in Theorem~\ref{thm:block-Ramsey} below, an extended, to every countable order,  Ramsey type partition theorem for variable $\omega$-located words over an infinite countable alphabet  dominated by an increasing sequence. It is an extension to every countable order $\xi$ of  Theorem~\ref{thm:block-Ramsey2} corresponding to the case $\xi=1$. Theorem~\ref{thm:block-Ramsey} extends Bergelson-Blass-Hindman's Ramsey type partition theorem (Theorem 5.1 in  \cite{BBH}) for located words over a finite alphabet, corresponding to the case of finite ordinals, and Furstenberg-Katznelson's (\cite{FuK}) Ramsey type partition theorem for words over a finite alphabet. 
 
Consequences of Theorem~\ref{thm:block-Ramsey} are multidimentional partition theorems for semigroups corresponding to each countable order $\xi$, providing strong simultaneous extension of the block-Ramsey partition theorem for every countable ordinal, proved in \cite{FN1}, and of van der Waerden theorem  \cite{vdW} for general semigroups. In Corollary~\ref{cor:multi} we present the case of finite ordinals, which imply the partition theorem proved by Beiglb$\ddot{o}$ck (Theorem 1.1 in \cite{Be}) for the particular case of commutative semigroups , and in Corollary~\ref{thm:van der Waerden1} we present the case $\xi=\omega$. 

The vehicle of proving this extended Ramsey type partition theorem (Theorem~\ref{thm:block-Ramsey}) is the Schreier system  $(L^\xi(\Sigma, \vec{k} ; \upsilon))_{\xi<\omega_1}$ (Definition~\ref{recursivethinblock}), consisted of families of finite orderly sequences of variable $\omega$-located words over the alphabet $\Sigma$ dominated by the sequence 
$\vec{k}$. Instrumental for this definition are the Schreier sets $\A_\xi$ consisted of finite subsets of $\nat$ defined initially in \cite{F1} and completelly in \cite{F3}, which are defined below employing (in case 3(iii)) the Cantor normal form of ordinals 
(cf. \cite{KM}, \cite{L}). 
\newline
We denote by $[X]^{<\omega}$ the set of all finite subsets and by $[X]^{<\omega}_{>0}$ the set of all non-empty, finite subsets of a set $X$. For $s_1, s_2 \in [\nat]^{<\omega}_{>0}$ we write $s_1 < s_2$ if $\max s_1 < \min s_2$.
\begin{defn}[\textsf{The Schreier system}, 
{[F1, Def. 7], [F2, Def. 1.5] [F3, Def. 1.4]}]
\label{Irecursivethin}
For every non-zero, countable, limit ordinal $\lambda$ choose and fix a strictly 
increasing sequence $(\lambda_n)_{n \in \nat}$ of successor ordinals smaller than $\lambda$ 
with $\sup_n \lambda_n = \lambda$. 
The system $(\A_\xi)_{\xi<\omega_1}$ is defined recursively as follows: 
\begin{itemize}
\item[(1)] $\A_0 = \{\emptyset\}$ and $\A_1 = \{\{n\} : n\in \nat\}$;
\item[(2)] $\A_{\zeta+1} = \{s\in [\nat]^{<\omega}_{>0} : s= \{n\} \cup s_1$, 
where $n\in \nat$, $\{n\} <s_1$ and $s_1\in \A_\zeta\}$; 
\item[(3i)] $\A_{\omega^{\beta+1}} = \{s\in [\nat]^{<\omega}_{>0} : 
s = \bigcup_{i=1}^n s_i$, where $n= \min s_1$, $s_1<\cdots < s_n$ and 
$s_1,\ldots, s_n\in \A_{\omega^\beta}\}$; 
\item[(3ii)] for a non-zero, countable limit ordinal $\lambda$, 
\newline
$\A_{\omega^\lambda} = \{s\in [\nat]^{<\omega}_{>0} : s\in \A_{\omega^{\lambda_n}}$
with $n= \min s\}$; and 
\item[(3iii)] for a limit ordinal $\xi$ such that $\omega^{\alpha}< \xi < \omega^{\alpha +1}$ for some $0< \alpha <\omega_1$, if \newline
$\xi = \omega^{\alpha} p
+ \sum_{i=1}^m \omega^{a_i} p_i$, where $m\in \nat$ with $m\ge0$, 
 $p,p_1,\ldots,p_m$ are natural numbers with $p,p_1,\ldots,p_m\ge1$ 
(so that either $p>1$, or $p=1$ and $m\ge 1$) and 
$a,a_1,\ldots,a_m$ are ordinals with $a>a_1>\cdots a_m >0$, 
\newline
$\A_\xi = \{s\in [\nat]^{<\omega}_{>0} :s= s_0 \cup (\bigcup_{i=1}^m s_i)$ 
with $s_m < \cdots < s_1 <s_0$,
$s_0= s_1^0\cup\cdots\cup s_p^0$ with $s_1^0<\cdots < s_p^0\in \A_{\omega^a}$,
and $s_i = s_1^i \cup\cdots\cup s_{p_i}^i$ with 
$s_1^i <\cdots<\ s_{p_i}^i\in \A_{\omega^{a_i}}$ $\forall\ 1\le i\le m\}$.
\end{itemize}
\end{defn} 

Let $\Sigma=\{\alpha_1, \alpha_2, \ldots \}$ be an infinite countable alphabet, $\upsilon \notin \Sigma$ a variable and $\vec{k}=(k_n)_{n\in\nat}\subseteq \nat$. We  define the \textit{finite orderly sequences of $\omega$-located words} over $\Sigma$ dominated by $\vec{k}$ as follows:
\begin{itemize}
\item[{}] $L^{<\infty}(\Sigma, \vec{k} ; \upsilon) = \{\bw = (w_1,\ldots,w_l) : l\in\nat, $ 
$w_1<\cdots < w_l\in L(\Sigma, \vec{k} ; \upsilon) \} \cup \{\emptyset \}$, 
\item[{}] $L^{<\infty}(\Sigma, \vec{k}) = \{\bw = (w_1,\ldots,w_l) : l\in\nat, $ 
$w_1<\cdots < w_l\in L(\Sigma, \vec{k}) \} \cup \{\emptyset \}$.
\end{itemize}

The Schreier system  $(L^\xi(\Sigma, \vec{k} ; \upsilon))_{\xi<\omega_1}$ is defined recursively as follows:
   
\begin{defn} 
[\textsf{The Schreier systems} $(L^\xi(\Sigma, \vec{k} ; \upsilon))_{\xi<\omega_1}$]
\label{recursivethinblock}
Let $\Sigma=\{\alpha_1, \alpha_2, \ldots \}$ be an infinite countable alphabet, ordered according to the natural numbers, $\upsilon \notin \Sigma$ a variable and $\vec{k}=(k_n)_{n\in\nat}\subseteq \nat$. We set:
\newline
$L^0(\Sigma, \vec{k} ; \upsilon)= \{\emptyset\}=L^0(\Sigma, \vec{k}) $, and
\newline
for every countable ordinal $\xi\ge1$, 
\newline
$L^\xi(\Sigma, \vec{k}; \upsilon) = \{(w_1,\ldots,w_l)\in L^{<\infty}(\Sigma, \vec{k}; \upsilon) : 
\{\min dom(w_1),\ldots, \min dom(w_l)\}\in\A_\xi \}$,
\newline
$L^\xi(\Sigma, \vec{k}) = \{(w_1,\ldots,w_l)\in L^{<\infty}(\Sigma, \vec{k}) : 
\{\min dom(w_1),\ldots, \min dom(w_l)\}\in\A_\xi \}$. 

\end{defn}

\begin{remark}\label{rem1.4} 
(i) $L^\xi(\Sigma, \vec{k}; \upsilon) \subseteq L^{<\infty}(\Sigma, \vec{k}; \upsilon)$ and
$\emptyset \notin L^\xi(\Sigma, \vec{k}; \upsilon)$ for every $\xi\ge1$.

(ii) $L^m(\Sigma, \vec{k} ; \upsilon) = \{(w_1,\ldots, w_m) : w_1<\cdots < w_m \in L(\Sigma, \vec{k} ; \upsilon)\}$ for $m\in \nat$. 

(iii) $L^\omega(\Sigma, \vec{k} ; \upsilon) = \{(w_1,\ldots, w_n)\in L^{<\infty}(\Sigma, \vec{k} ; \upsilon) : n\in \nat$, 
and $\min dom(w_1) = n\}$. 
\end{remark} 

The following proposition justifies the recursiveness of the system $(L^\xi(\Sigma, \vec{k} ; \upsilon))_{\xi<\omega_1}$. 
\newline
For a family $\F \subseteq L^{<\infty} (\Sigma\cup\{\upsilon\}, \vec{k})$ and $t\in L(\Sigma\cup\{\upsilon\}, \vec{k})$, we set

\begin{center}
$\F(t) = \{\bw\in L^{<\infty} (\Sigma\cup\{\upsilon\}, \vec{k}) : $ either $\bw=(w_1,\ldots,w_l)\neq \emptyset$ and $(t,w_1,w_2, \ldots,w_l) \in \F$ 
or $\bw=\emptyset$ and $(t) \in\F\}$, 
\end{center}
$\F -t = \{\bw \in \F :$ either $\bw = (w_1,\ldots,w_l)\neq \emptyset$ and $t <w_1$, or 
$\bw =\emptyset\}$. 

\begin{prop}\label{justification}
For every countable ordinal $\xi\ge 1$, there exists a concrete sequence 
$(\xi_n)_{n\in\nat}$ of countable ordinals with $\xi_n<\xi$ such that for $\Sigma=\{\alpha_1, \alpha_2, \ldots \}$ an infinite countable alphabet, $\upsilon \notin \Sigma$ a variable, $\vec{k}=(k_n)_{n\in\nat}\subseteq \nat$ and $t\in L(\Sigma, \vec{k}; \upsilon)$, with $\min dom(t)=n$, 
\begin{center}
$L^\xi(\Sigma, \vec{k}; \upsilon) (t) = L^{\xi_n}(\Sigma, \vec{k}; \upsilon)\cap (L^{<\infty}(\Sigma, \vec{k}; \upsilon) - t)$.
\end{center}

Moreover, $\xi_n =\zeta$ for every $n\in\nat$ if $\xi = \zeta+1$, and 
$(\xi_n)_{n\in\nat}$ is a strictly increasing sequence with $\sup_n \xi_n=\xi$ if $\xi$ 
is a limit ordinal.
\end{prop} 

\begin{proof}
It follows from Theorem~1.6 in \cite{F3}, according to which for every countable ordinal $\xi >0$ there exists a concrete sequence 
$(\xi_n)_{n\in\nat}$ of countable ordinals with  $\xi_n <\xi$, such that $\A_\xi (n) = \A_{\xi_n} \cap [\{n+1,n+2,\ldots\}]^{<\omega}$ for every $ n\in \nat$, where, 
\newline
$\A_\xi(n) = \{s\in [\nat]^{<\omega} : s \in [\nat]^{<\omega}_{>0}, n < \min s$ and $\{n\}\cup s\in \A_\xi$ or $s=\emptyset$ and $\{n\}\in \A_\xi \}$.
\newline
Moreover, $\xi_n = \zeta$ for every $n\in\nat$ if $\xi = \zeta+1$, and 
$(\xi_n)_{n\in\nat}$ is a strictly increasing sequence with $\sup_n \xi_n=\xi$ if $\xi$ 
is a limit ordinal. 
\end{proof}

In order to state and prove the principal result of this section, a Ramsey type partition theorem for $\omega$-located words extended to every countable order, we need the following notation:

\subsection*{Notation}
Let $\Sigma=\{\alpha_1, \alpha_2, \ldots \}$ be an infinite countable alphabet, ordered according to the natural numbers, $\upsilon \notin \Sigma$ a variable and  $\vec{k}=(k_n)_{n\in\nat}\subseteq \nat$ an increasing sequence. For $\vec{w}=(w_n)_{n\in\nat} \in L^\infty (\Sigma, \vec{k} ; \upsilon)$, $\bw = (w_1,\ldots,w_l)\in L^{<\infty}(\Sigma, \vec{k} ; \upsilon)$ and $t\in L(\Sigma, \vec{k} ; \upsilon)$, we set:

$EV^{<\infty}(\vec{w}) = \{\bu=(u_1,\ldots,u_l)\in L^{<\infty} (\Sigma, \vec{k} ; \upsilon) : l\in\nat, u_1,\ldots,u_l\in EV(\vec{w})\}\cup \{\emptyset \}$; 

$E^{<\infty}(\vec{w}) = \{\bu=(u_1,\ldots,u_l)\in L^{<\infty} (\Sigma, \vec{k}) : l\in\nat, u_1,\ldots,u_l\in E(\vec{w})\}\cup \{\emptyset \}$;   

$EV(\bw)=\{T_{p_1}(w_{n_1})\star \ldots \star T_{p_\lambda}(w_{n_\lambda})\in L(\Sigma, \vec{k} ; \upsilon)  : 1\leq n_1<\cdots<n_\lambda\leq l$ and
\begin{center}
$p_1,\ldots,p_\lambda\in\nat\cup\{0\}$ with $0\leq p_i \leq k_{n_i}$ for $1\leq i \leq \lambda$ and $0 \in \{ p_1,\ldots,p_\lambda\} \}$;
\end{center}

$E(\bw)=\{T_{p_1}(w_{n_1})\star \ldots \star T_{p_\lambda}(w_{n_\lambda})\in L(\Sigma, \vec{k} ; \upsilon)  : 1\leq n_1<\cdots<n_\lambda\leq l$ and
\begin{center}
$p_1,\ldots,p_\lambda\in\nat$ with $1\leq p_i \leq k_{n_i}$ for $1\leq i \leq \lambda \}$.
\end{center}
$EV^{<\infty}(\bw) = \{\bu=(u_1,\ldots,u_l)\in L^{<\infty} (\Sigma, \vec{k} ; \upsilon) : l\in\nat, u_1,\ldots,u_l\in EV(\bw)\}\cup \{\emptyset \}$.
\newline
Observe that the sets $EV(\bw)$, $E(\bw)$ are finite. Also, we set

$\vec{w} -t = (w_n)_{n\geq l}\in L^\infty (\Sigma, \vec{k} ; \upsilon)$, where $l=\min \{n\in\nat : t<w_n\}$, and

$\vec{w}-\bw =  \vec{w} -w_l$.

\begin{thm}
[\textsf{Ramsey type partition theorem on Schreier families for variable $\omega$-located words }]
\label{thm:block-Ramsey}
Let $\xi\geq 1$ be a countable ordinal, $\Sigma=\{\alpha_1, \alpha_2, \ldots \}$ be an infinite countable alphabet, $\upsilon \notin \Sigma$ a variable and  $\vec{k}=(k_n)_{n\in\nat}\subseteq \nat$ an increasing sequence. For every family $\F \subseteq L^{<\infty}(\Sigma, \vec{k} ; \upsilon)$ of finite orderly sequences of variable $\omega$-located words and every infinite orderly sequence $\vec{w} \in L^\infty (\Sigma, \vec{k} ; \upsilon)$ of variable $\omega$-located words there exists an extraction $\vec{u}\prec \vec{w}$ of $\vec{w}$ over $\Sigma$ such that 

 either $L^\xi(\Sigma, \vec{k} ; \upsilon) \cap EV^{<\infty}(\vec{u})\subseteq \F$, or 
$L^\xi(\Sigma, \vec{k} ; \upsilon) \cap EV^{<\infty}(\vec{u})\subseteq L^{<\infty}(\Sigma, \vec{k} ; \upsilon)\setminus \F$. 
\end{thm}

In the proof of Theorem~\ref{thm:block-Ramsey} we will make use of the following diagonal argument.  

\begin{lem}\label{lem:block-Ramsey}
Let $\Sigma=\{\alpha_1, \alpha_2, \ldots \}$ be an infinite countable alphabet, $\upsilon \notin \Sigma$ a variable, $\vec{k}=(k_n)_{n\in\nat}\subseteq \nat$ an increasing sequence, $\vec{w} = (w_n)_{n\in\nat} \in L^\infty (\Sigma, \vec{k} ; \upsilon)$, 
\newline
$\Pi = \{(t,\vec{s}): t\in L(\Sigma, \vec{k} ; \upsilon)$, 
$\vec{s} = (s_n)_{n\in\nat}\in L^\infty (\Sigma, \vec{k} ; \upsilon)$ with $\vec{s}\prec \vec{w}$ and $t<s_n \forall\ n\in\nat\}$.
\newline
If a subset $\R$ of $\Pi$ satisfies
\begin{itemize}
\item[(i)] for every $(t,\vec{s})\in\Pi$, there exists $(t,\vec{s}_1)\in\R$ with 
$\vec{s}_1 \prec \vec{s}$; and
\item[(ii)] for every $(t,\vec{s}) \in\R$ and $\vec{s}_1 \prec \vec{s}$, we have $(t,\vec{s}_1)\in\R$,
\end{itemize}
then there exists $\vec{u} \prec \vec{w}$, such that 
$(t,\vec{s})\in \R$ for all $t\in EV(\vec{u})$ and $\vec{s} \prec \vec{u}-t$.
\end{lem}

\begin{proof} 
Let $u_0 = w_1$. 
According to condition (i), there exists $\vec{s}_1 = (s^1_n)_{n\in\nat} \in 
L^\infty (\Sigma, \vec{k} ; \upsilon)$ 
with $\vec{s}_1\prec\vec{w}-u_0$ such that 
$(u_0,\vec{s}_1)\in \R$. 
Let $u_1 = s^1_1$. 
Of course, $u_0<u_1$ and $u_0, u_1 \in EV(\vec{w})$. 
We assume now that there have been constructed 
$\vec{s}_1,\ldots,\vec{s}_n \in L^\infty (\Sigma, \vec{k} ; \upsilon)$ and $u_0,u_1,\ldots,u_n \in EV(\vec{w})$, with 
$\vec{s}_n \prec \cdots \prec \vec{s}_1 \prec \vec{w}$, $u_0<u_1<\cdots < u_n$ and 
$(t,\vec{s}_i)\in\R$ for all $1\le i\le n$ and $t\in EV((u_0,\ldots,u_{i-1}))$.

We will construct $\vec{s}_{n+1}$ and $u_{n+1}$. 
Let $\{ t_1,\ldots, t_l\} = EV ((u_0,\ldots, u_n))$. 
According to condition (i), there exist 
$\vec{s}_{n+1}^1,\ldots, \vec{s}_{n+1}^l \in L^\infty (\Sigma, \vec{k} ; \upsilon)$ such that 
$\vec{s}_{n+1}^l \prec\cdots \prec \vec{s}_{n+1}^1 \prec \vec{s}_n -u_n$ and 
$(t_i, \vec{s}_{n+1}^i)\in \R$ for every $1\le i\le l$. 
Set $\vec{s}_{n+1} = \vec{s}_{n+1}^l$. If $\vec{s}_{n+1} = (s^{n+1}_n)_{n\in\nat}$, set $u_{n+1} = s^{n+1}_1$.  
Of course $u_n < u_{n+1}$, $u_{n+1} \in EV(\vec{w})$ and, according to condition (ii), 
$(t_i, \vec{s}_{n+1})\in \R$ for all $1\le i\le l$. 

Set $\vec{u} = (u_0,u_1,u_2,\ldots) \in L^\infty (\Sigma, \vec{k} ; \upsilon)$. Then 
$\vec{u} \prec \vec{w}$, since $u_0<u_1< \ldots \in EV(\vec{w})$. 
Let $t\in EV(\vec{u})$ and $\vec{s} \prec \vec{u}-t$. 
Set $n_0 = \min \{n\in\nat  : t\in EV((u_0,u_1,\ldots,u_n))\}$. Since $t\in EV((u_0,u_1,\ldots,u_{n_0}))$, 
we have $(t,\vec{s}_{n_0+1})\in \R$. Then, according to (ii), we have $(t,\vec{u}-u_{n_0})\in\R$, since 
$\vec{u}-u_{n_0}\prec \vec{s}_{n_0+1}$, and also $(t,\vec{s}) \in\R$, since $\vec{s}\prec\vec{u}-u_{n_0}=\vec{u}-t$.
\end{proof}

\begin{proof}[Proof of Theorem~\ref{thm:block-Ramsey}] 
Let $\F\subseteq L^{<\infty}(\Sigma, \vec{k} ; \upsilon)$ and $\vec{w} \in L^\infty (\Sigma, \vec{k} ; \upsilon)$. For $\xi=1$ the theorem is valid, according to Theorem~\ref{thm:block-Ramsey2}. Let $\xi>1$. Assume that the theorem is valid for every $\zeta <\xi$. Let $t\in L(\Sigma, \vec{k} ; \upsilon)$ with $\min dom(t)=n$ and $\vec{s} = (s_n)_{n\in\nat}\in L^\infty (\Sigma, \vec{k} ; \upsilon)$ with $\vec{s}\prec \vec{w}$ and $t<s_n$ for all $ n\in\nat$. According to Proposition~\ref{justification}, there exists $\xi_n<\xi$ such that 
\begin{center}
$L^\xi(\Sigma, \vec{k}; \upsilon) (t) = L^{\xi_n}(\Sigma, \vec{k}; \upsilon)\cap (L^{<\infty}(\Sigma, \vec{k}; \upsilon) - t)$.
\end{center}
Using the induction hypothesis, there exists $\vec{s}_1 \prec\vec{s}$ such that 
\newline
either $L^{\xi_n}(\Sigma, \vec{k} ; \upsilon) \cap EV^{<\infty}(\vec{s}_1)\subseteq \F(t)$, or 
$L^{\xi_n}(\Sigma, \vec{k} ; \upsilon) \cap EV^{<\infty}(\vec{s}_1)\subseteq L^{<\infty}(\Sigma, \vec{k} ; \upsilon)\setminus \F(t)$.
\newline
Then $\vec{s}_1 \prec \vec{s} \prec \vec{w}$, and  

either $L^{\xi}(\Sigma, \vec{k} ; \upsilon)(t) \cap EV^{<\infty}(\vec{s}_1)\subseteq \F(t)$, 

or $L^{\xi}(\Sigma, \vec{k} ; \upsilon)(t) \cap EV^{<\infty}(\vec{s}_1)\subseteq L^{<\infty}(\Sigma, \vec{k} ; \upsilon)\setminus \F(t)$.
\newline
Let $\R= \{(t,\vec{s}): t\in L(\Sigma, \vec{k} ; \upsilon)$, 
$\vec{s} = (s_n)_{n\in\nat}\in L^\infty (\Sigma, \vec{k} ; \upsilon),\vec{s}\prec \vec{w}$, $t<s_n \forall\ n\in\nat,$ and 

either $L^{\xi}(\Sigma, \vec{k} ; \upsilon)(t) \cap EV^{<\infty}(\vec{s})\subseteq \F(t)$ 

or $L^{\xi}(\Sigma,\vec{k}; \upsilon)(t)\cap EV^{<\infty}(\vec{s})\subseteq L^{<\infty}(\Sigma,\vec{k}; \upsilon)\setminus\F(t)\}$.
\newline
The family $\R$ satisfies the conditions (i) (by the above arguments) and 
(ii) (obviously) of Lemma~\ref{lem:block-Ramsey}. 
Hence, there exists $\vec{u}_1 \prec \vec{w}$ such that 
$(t,\vec{s})\in \R$ for all $t\in EV(\vec{u}_1)$ and $\vec{s} \prec \vec{u}_1-t$.
\newline
Let $\F_1 = \{t\in EV(\vec{u}_1) : L^{\xi}(\Sigma, \vec{k} ; \upsilon)(t) \cap EV^{<\infty}(\vec{u}_1-t)\subseteq \F(t)\}$. 
\newline
We use the induction hypothesis for $\xi=1$ (Theorem~\ref{thm:block-Ramsey2}). 
Then there exists a variable extraction $\vec{u} \prec \vec{u}_1$ of $\vec{u}_1$ such that  

 either $EV(\vec{u})\subseteq \F_1$, or 
$EV(\vec{u})\subseteq L(\Sigma, \vec{k} ; \upsilon)\setminus \F_1$.
\newline
Since $\vec{u} \prec \vec{u}_1$ we have that $EV(\vec{u})\subseteq EV(\vec{u}_1)$, 
and, consequently, that $(t,\vec{u}-t)\in \R$ for all $t\in EV(\vec{u})$. Thus 

 either $L^{\xi}(\Sigma, \vec{k} ; \upsilon)(t) \cap EV^{<\infty}(\vec{u}-t)\subseteq \F(t)$ 
for all $t\in EV(\vec{u})$,

 or $L^{\xi}(\Sigma, \vec{k}; \upsilon )(t) \cap EV^{<\infty}(\vec{u}-t)\subseteq L^{<\infty}(\Sigma, \vec{k} ; \upsilon)\setminus \F(t)$ for all $t\in EV(\vec{u})$.
\newline
Hence,

 either $L^\xi(\Sigma, \vec{k} ; \upsilon) \cap EV^{<\infty}(\vec{u})\subseteq \F$, or 
$L^\xi(\Sigma, \vec{k} ; \upsilon) \cap EV^{<\infty}(\vec{u})\subseteq L^{<\infty}(\Sigma, \vec{k} ; \upsilon)\setminus \F$.
\end{proof}
\begin{remark}\label{rem. D}
(i) The particular case of Theorem~\ref{thm:block-Ramsey} for $\vec{k}=(k_n)_{n\in\nat}\subseteq \nat$ with $k_n=k_1$ for every $n\in\nat$ gives an extended to every countable ordinal $\xi$ Ramsey type partition theorem for variable located words over a finite alphabet, which in turn contains Bergelson-Blass-Hindman's Ramsey type partition theorem (Theorem 5.1 in  \cite{BBH}) as the special case $\xi$ a finite ordinal.
\newline
(ii) The case $k_n=1$ for every $n\in\nat$ of Theorem~\ref{thm:block-Ramsey} implies the block Ramsey partition theorem for every countable ordinal $\xi$ proved in \cite{FN1} (Theorem 1.6), which in turn contains Milliken-Taylor's partition theorem (\cite{M}, \cite{T}), as the special case $\xi<\omega$.
\end{remark}
Analogously to Theorem~\ref{thm:block-Ramsey} can be proved the following extended to every countable order  Ramsey type partition theorem for $\omega$-located words. 
\begin{thm}
[\textsf{Ramsey type partition theorem on Schreier families for $\omega$-located words }]
\label{thm:block-Ramsey8}
Let $\xi\geq 1$ be a countable ordinal, $\Sigma=\{\alpha_1, \alpha_2, \ldots \}$ be an infinite countable alphabet, $\upsilon \notin \Sigma$ a variable and  $\vec{k}=(k_n)_{n\in\nat}\subseteq \nat$ an increasing sequence. For every family $\G \subseteq L^{<\infty}(\Sigma, \vec{k})$ of finite orderly sequences of $\omega$-located words and every infinite orderly sequence $\vec{w} \in L^\infty (\Sigma, \vec{k} ; \upsilon)$ of variable $\omega$-located words there exists an extraction $\vec{u}\prec \vec{w}$ of $\vec{w}$ over $\Sigma$ such that  either $L^\xi(\Sigma, \vec{k} ) \cap E^{<\infty}(\vec{u})\subseteq \G$, or 
$L^\xi(\Sigma, \vec{k}) \cap E^{<\infty}(\vec{u})\subseteq L^{<\infty}(\Sigma, \vec{k})\setminus \G$. 
\end{thm}

\begin{cor}
[\textsf{Ramsey type partition theorem for $\omega$-located words}] 
\label{cor:k-Ramsey} 
Let $m\in\nat$, $\Sigma=\{\alpha_1, \alpha_2, \ldots \}$ be an infinite countable alphabet, $\upsilon \notin \Sigma$ a variable, $\vec{k}=(k_n)_{n\in\nat}\subseteq \nat$ an increasing sequence, $r,s\in\nat$ and $\vec{w} \in L^\infty (\Sigma, \vec{k} ; \upsilon)$. If $L^m(\Sigma, \vec{k} ; \upsilon)=A_1\cup\cdots\cup A_r$ and $L^m(\Sigma, \vec{k})=C_1\cup\cdots\cup C_s$, then there exist an extraction $\vec{u}\prec \vec{w}$ of $\vec{w}$ over $\Sigma$ and $1\leq i_0 \leq r$, $1\leq j_0 \leq s$ such that
\begin{center}
$ \{(z_1,\ldots,z_m)\in L^{<\infty}(\Sigma, \vec{k} ; \upsilon)) : z_1,\ldots,z_m\in EV(\vec{u})\}\subseteq A_{i_0}$, and
\end{center} 
\begin{center}
$ \{(z_1,\ldots,z_m)\in L^{<\infty}(\Sigma, \vec{k}) : z_1,\ldots,z_m\in E(\vec{u})\}\subseteq C_{j_0}$.
\end{center}
\end{cor}
From Corollary~\ref{cor:k-Ramsey} we can derive, for every $m\in\nat$, a strong m-dimensional partition theorem for semigroups, via the function  $g:L(\nat,\vec{k})\rightarrow X$ with 
$g(w_{n_1}\ldots w_{n_l})=\sum^{l}_{i=1}y_{w_{n_i},\;n_i}$ (see the remarks before Corollary~\ref{thm:van der Waerden}). In case of commutative semigroups an analogous result has been proved by Beiglb$\ddot{o}$ck  (Theorem 1.1 in \cite{Be}). For a set $X$ we denote by $[X]^{m}$ the set of all the subsets of $X$ with exactly $m$ elements.
\begin{cor}
\label{cor:multi}
Let $(X,+)$ be a semigroup, $\vec{k}=(k_n)_{n \in \nat}\subseteq \nat$ an increasing sequence, $(y_{l,n})_{n \in \nat}$ for every $l \in \nat,$ sequences in $X$ and $m\in\nat$. If $[X]^{m}=A_1\cup \ldots \cup A_r,\; r \in \nat$, then there exists $\vec{w}=(w_n)_{n\in\nat} \in L^\infty (\nat, \vec{k} ; \upsilon)$ with  $w^{n}_{q^{n}_1}, w^{n}_{q^{n}_{l_n}}\in \nat$ if $w_n=w^{n}_{q^{n}_1}\ldots w^{n}_{q^{n}_{l_n}}$ for every $n \in \nat$ and there exists $1\leq i_0\leq r$ such that 
\begin{center}
$ \{(g(z_1),\ldots,g(z_m))\in [X]^{m} :(z_1,\ldots,z_m)\in L^{<\infty}(\nat, \vec{k}), z_1,\ldots,z_m\in E(\vec{w})\}\subseteq A_{i_0}$.
\end{center}
\end{cor}

\begin{cor}
[\textsf{$\omega$-Ramsey type partition theorem for $\omega$-located words}] 
\label{cor:kk-Ramsey} 
Let $\Sigma=\{\alpha_1, \alpha_2, \ldots \}$ be an infinite countable alphabet, $\upsilon \notin \Sigma$ a variable, $\vec{k}=(k_n)_{n\in\nat}\subseteq \nat$ an increasing sequence, $r,s\in\nat$ and $\vec{w} \in L^\infty (\Sigma, \vec{k} ; \upsilon)$. If $L^{<\infty}(\Sigma, \vec{k} ; \upsilon)=A_1\cup\cdots\cup A_r$ and $L^{<\infty}(\Sigma, \vec{k})=C_1\cup\cdots\cup C_s$, then there exist an extraction $\vec{u}\prec \vec{w}$ of $\vec{w}$ over $\Sigma$ and $1\leq i_0 \leq r$, $1\leq j_0 \leq s$ such that
\begin{center}
$ \{(z_1,\ldots,z_n)\in L^{<\infty}(\Sigma, \vec{k} ; \upsilon)) : n\in \nat, \min dom(z_1) = n $ and $z_1,\ldots,z_n\in EV(\vec{u})\}\subseteq A_{i_0}$,
\end{center} 
\begin{center}
$ \{(z_1,\ldots,z_n)\in L^{<\infty}(\Sigma, \vec{k})) : n\in \nat, \min dom(z_1) = n $ and $z_1,\ldots,z_n\in E(\vec{u})\}\subseteq C_{j_0}$.
\end{center}
\end{cor}
As a consequence of Corollary~\ref{cor:kk-Ramsey} we have the following. 

\begin{cor}
\label{thm:van der Waerden1}
Let $(X,+)$ be a semigroup, $\vec{k}=(k_n)_{n \in \nat}\subseteq \nat$ an increasing sequence and $(y_{l,n})_{n \in \nat}$ for every $l \in \nat,$ sequences in $X$. If $[X]^{<\omega}_{>0}=A_1\cup \ldots \cup A_r,\; r \in \nat$, then there exists $\vec{w}=(w_n)_{n\in\nat} \in L^\infty(\nat, \vec{k} ; \upsilon)$ with  $w^{n}_{q^{n}_1}, w^{n}_{q^{n}_{l_n}}\in \nat$ if $w_n=w^{n}_{q^{n}_1}\ldots w^{n}_{q^{n}_{l_n}}$ for every $n \in \nat$ and there exists $1\leq i_0\leq r$ such that 
\begin{center}
$ \{(g(z_1),\ldots,g(z_n))\in [X]^{<\omega}_{>0} :n\in \nat, (z_1,\ldots,z_n)\in L^{<\infty}(\Sigma, \vec{k})$ with $\min dom(z_1)=n, z_1,\ldots,z_n\in E(\vec{w})\}\subseteq A_{i_0}$.
\end{center}
\end{cor}
We will give now a Ramsey type partition theorem for unlocated $\omega$-words, as a corollary of Theorem~\ref{thm:block-Ramsey}, extending Furstenberg and Katznelson's partition theorem ( \cite{FuK}), for words over a finite alphabet. 

Let $m\in\nat$, $\Sigma=\{\alpha_1, \alpha_2, \ldots \}$ an infinite countable alphabet, $\vec{k}=(k_n)_{n\in\nat}\subseteq \nat$ an increasing sequence, $\upsilon \notin \Sigma$ a variable and $\vec{w}=(w_n)_{n\in\nat} \in W^\infty (\Sigma, \vec{k} ; \upsilon)$.  
\begin{itemize}
\item[{}] $W^{m}(\Sigma, \vec{k}; \upsilon) = \{\bu = (u_1,\ldots,u_m) : u_1<\cdots < u_m\in W(\Sigma, \vec{k}; \upsilon) \} $, and 
\item[{}] $W^{m}(\Sigma, \vec{k}) = \{\bu = (u_1,\ldots,u_m) : u_1<\cdots < u_m\in W(\Sigma, \vec{k})\}$. 
\end{itemize}
An element $\bu = (u_1,\ldots,u_m)$ of $W^{m}(\Sigma, \vec{k}; \upsilon)$ is a \textit{ reduced m-sequence of variable $\omega$-words} of $\vec{w}$ if there exist $0=\lambda_1<\lambda_2<\cdots<\lambda_{m+1}\in \nat\cup\{0\}$ such that for every $1\leq i \leq m$
\begin{center}
$u_i=T_{p_{\lambda_i +1}}(w_{\lambda_i +1})\star \ldots \star T_{p_{\lambda_{i+1}}}(w_{\lambda_{i+1}})$, 
\end{center}
where, for $\lambda_i +1\leq j\leq \lambda_{i+1}$, $p_j\in\nat\cup\{0\}$, $0\leq p_j \leq k_{j}$ and 
$0 \in \{ p_{\lambda_i+1},\ldots,p_{\lambda_{i+1}}\}$. 
\newline
The set of all the reduced m-sequences of variable $\omega$-words of $\vec{w}$ is denoted by $RV^m(\vec{w})$.

An element $\bu = (u_1,\ldots,u_m)$ of $W^{m}(\Sigma, \vec{k})$ is a \textit{ reduced m-sequence of  $\omega$-words} of $\vec{w}$ if there exist $0=\lambda_1<\lambda_2<\cdots<\lambda_{m+1}\in \nat\cup\{0\}$ such that for every $1\leq i \leq m$
\begin{center}
$u_i=T_{p_{\lambda_i +1}}(w_{\lambda_i +1})\star \ldots \star T_{p_{\lambda_{i+1}}}(w_{\lambda_{i+1}})$, 
\end{center}
where, $p_j\in\nat$, $1\leq p_j \leq k_{j}$ for every $\lambda_i +1\leq j\leq \lambda_{i+1}$. 
\newline
The set of all the reduced m-sequences of variable $\omega$-words of $\vec{w}$ is denoted by $R^m(\vec{w})$.

\begin{thm}
[\textsf{Ramsey type partition theorem for $\omega$-words}] 
\label{cor:k-Ramsey words} 
Let $m\in\nat$, $\Sigma=\{\alpha_1, \alpha_2, \ldots \}$ be an infinite countable alphabet, $\upsilon \notin \Sigma$ a variable, $\vec{k}=(k_n)_{n\in\nat}\subseteq \nat$ an increasing sequence, and $r,s\in\nat$. If $W^m(\Sigma, \vec{k} ; \upsilon)=A_1\cup\cdots\cup A_r$ and $W^m(\Sigma, \vec{k})=C_1\cup\cdots\cup C_s$, then there exist and $\vec{w}=(w_n)_{n\in\nat} \in W^\infty (\Sigma, \vec{k} ; \upsilon)$ and $1\leq i_0 \leq r$, $1\leq j_0 \leq s$ such that 
\begin{center}
$RV^m(\vec{w})\subseteq A_{i_0}$ and $R^m(\vec{w})\subseteq C_{j_0}$.
\end{center}
\end{thm}
\begin{proof}
Let the function $f : L(\Sigma\cup\{\upsilon\}, \vec{k}) \longrightarrow W(\Sigma\cup\{\upsilon\}, \vec{k})$ with $f(w_{n_1}\ldots w_{n_l})=u_1\ldots u_{n_l}\in W(\Sigma\cup\{\upsilon\}, \vec{k})$ where $u_{n_j}=w_{n_j}$ for every $1\leq j\leq l$ and $u_i=\alpha_1$ for every $i\in\{1,\ldots, n_l\}\setminus \{n_j : 1\leq j\leq l \}$. We define $\tilde{f} : L^m(\Sigma, \vec{k} ; \upsilon)\cup L^{m}(\Sigma, \vec{k}) \longrightarrow 
W^m(\Sigma, \vec{k}; \upsilon)\cup W^{m}(\Sigma, \vec{k})$ with $\tilde{f}( (w_1,\ldots,w_m) )=(f(w_1),\ldots,f(w_m))$. 

According to Corollary~\ref{cor:k-Ramsey} there exist an infinite sequence $\vec{s}=(s_n)_{n\in\nat}\in L^\infty (\Sigma, \vec{k} ; \upsilon)$ and $1\leq i_0 \leq r$, $1\leq j_0 \leq s$ such that 
\begin{center}
$ \{(z_1,\ldots,z_m)\in L^{<\infty}(\Sigma, \vec{k} ; \upsilon)) : z_1,\ldots,z_m\in EV(\vec{s})\}\subseteq \tilde{f}^{-1}(A_{i_0})$, and
\end{center} 
\begin{center}
$ \{(z_1,\ldots,z_m)\in L^{<\infty}(\Sigma, \vec{k}) : z_1,\ldots,z_m\in E(\vec{s})\}\subseteq \tilde{f}^{-1}(C_{j_0})$.
\end{center}
Set $w_n=f(s_n)$ for every $n\in\nat$ and $\vec{w}=(w_n)_{n\in\nat}\in W^{\infty}(\Sigma, \vec{k} ; \upsilon)$. Then $RV^m(\vec{w})\subseteq A_{i_0}$ and $R^m(\vec{w})\subseteq C_{j_0}$. 
\end{proof}

\section{Partition theorems for sequences of variable $\omega$-located words }

The main result of this Section is Theorem~\ref{block-NashWilliams2}, which strengthens the extended Ramsey type partition theorem for variable $\omega$-located words  (Theorem~\ref{thm:block-Ramsey}) in case the partition family is a tree. Specificaly, given a partition family $\F \subseteq L^{<\infty}(\Sigma, \vec{k} ; \upsilon)$ of finite orderly sequences of variable $\omega$-located words over an alphabet $\Sigma=\{\alpha_1, \alpha_2, \ldots \}$ dominated by a sequence $\vec{k}=(k_n)_{n\in\nat}\subseteq \nat$ and $\xi<\omega_1$,  Theorem~\ref{thm:block-Ramsey} provides no information on how to decide whether the homogeneous family falls in $\F$ or in its complement $L^{<\infty}(\Sigma, \vec{k} ; \upsilon)\setminus \F$, while Theorem~\ref{block-NashWilliams2} in case the partition family $\F$ is a tree provides a criterion, in terms of a Cantor-Bendixson type index of $\F$, according to which we can have such a decition. 

Using Theorem~\ref{block-NashWilliams2} we obtain a partition theorem for infinite orderly sequences of variable $\omega$-located words (Theorem~\ref{cor:blockNW}), which, can be said to be a Nash-Williams type partition theorem for variable $\omega$-located words. Particular cases of Theorem~\ref{cor:blockNW} are Bergelson-Blass-Hindman's in \cite{BBH} (Theorem 6.1) for variable located words over a finite alphabet and 
Carlson's theorem (Theorem 2 in \cite{C}) for variable words over a finite alphabet. 

As a consequence of Theorem~\ref{cor:blockNW} we prove, in Theorems~\ref{cor:centralNW} and ~\ref{cor:ncentralNW}, partition theorems for infinite sequences in a commutative and in a noncommutative semigroup, respectively, which are strong simultaneous extentions of the infinitary partition theorem of Milliken \cite{M}, Taylor \cite{T} and van der Waerden \cite{vdW} for general semigroups.

\subsection*{Notation}
Let $\Sigma=\{\alpha_1, \alpha_2, \ldots \}$ be an infinite countable alphabet, $\upsilon \notin \Sigma$ a variable and $\vec{k}=(k_n)_{n\in\nat}\subseteq \nat$. A finite orderly sequence $\bw=(w_1,\ldots,w_l)\in L^{<\infty}(\Sigma, \vec{k} ; \upsilon)$ is an \textit{initial segment} of $\bu=(u_1,\ldots,u_k)\in L^{<\infty}(\Sigma, \vec{k} ; \upsilon)$ iff $l\leq k$ and $w_i= u_i$ for  every $i=1,\ldots,l$ and $\bw$ is an \textit{initial segment} of $\vec{u}=(u_n)_{n\in\nat}\in L^{\infty}(\Sigma, \vec{k} ; \upsilon)$ if $w_i=u_i$ for all $i=1,\ldots,l$. In these cases we write $\bw\propto \bu$ and  $\bw\propto \vec{u}$, respectively, and we set $\bu\setminus \bw = (u_{l+1},\ldots,u_k)$ and $\vec{u}\setminus \bw = (u_n)_{n>l}$.

\begin{defn}\label{def:Fthin}
Let $\Sigma=\{\alpha_1, \alpha_2, \ldots \}$ be an infinite countable alphabet, $\upsilon \notin \Sigma$ a variable, $\vec{k}=(k_n)_{n\in\nat}\subseteq \nat$ and $\F\subseteq L^{<\infty}(\Sigma, \vec{k} ; \upsilon)$.
\begin{itemize}
\item[(i)] $\F$ is {\em thin\/} if there are no elements $\bw,\bu\in\F$
with $\bw \propto \bu$ and $\bw\ne \bu$.
\item[(ii)] $\F^* = \{\bw \in L^{<\infty}(\Sigma, \vec{k}; \upsilon): \bw\propto \bu$ for some 
$\bu\in \F\}\cup \{\emptyset\}$.
\item[(iii)] $\F$ is a {\em tree\/} if $\F^* = \F$.
\item[(iv)] $\F_* = \{\bw \in L^{<\infty}(\Sigma, \vec{k}; \upsilon): \bw\in EV^{<\infty}(\bu)$ 
for some $\bu\in \F\}\cup\{\emptyset\}$.
\item[(v)] $\F$ is {\em hereditary\/} if $\F_* = \F$.
\end{itemize}
\end{defn}

\begin{prop}\label{prop:thinfamily}
Every family $L^\xi(\Sigma, \vec{k} ; \upsilon)$, for $\xi<\omega_1$ is thin.
\end{prop} 
\begin{proof} 
It follows from the fact that the families $\A_\xi$ are thin (cf. \cite{F3})(which means that if 
$s,t\in \A_\xi$ and $s\propto t$, then $s=t$).
\end{proof}

\begin{prop}\label{prop:canonicalrep}
Let $\xi$ be a nonzero countable ordinal number, $\Sigma=\{\alpha_1, \alpha_2, \ldots \}$ be an infinite countable alphabet, $\upsilon \notin \Sigma$ a variable and  $\vec{k}=(k_n)_{n\in\nat}\subseteq \nat$. Then

(i) every infinite orderly sequence  $\vec{s} = (s_n)_{n\in\nat}\in L^{\infty}(\Sigma, \vec{k} ; \upsilon)$ has canonical representation with respect to $L^\xi(\Sigma, \vec{k} ; \upsilon)$, which means that there exists a unique strictly increasing sequence $(m_n)_{n\in\nat}$ in $\nat$ 
so that $(s_1,\ldots,s_{m_1}) \in L^\xi(\Sigma, \vec{k} ; \upsilon)$ and $(s_{m_{n-1}+1},\ldots,s_{m_n}) \in L^\xi(\Sigma, \vec{k} ; \upsilon)$ for every $n > 1$; and,  

(ii) every nonempty finite orderly sequence $\bs = (s_1,\ldots,s_k)\in L^{<\infty}(\Sigma, \vec{k} ; \upsilon)$ has canonical representation with respect to $L^\xi(\Sigma, \vec{k} ; \upsilon)$, so either $\bs\in (L^\xi(\Sigma, \vec{k} ; \upsilon))^* \setminus L^\xi(\Sigma, \vec{k} ; \upsilon)$ or there exist unique $n\in\nat$, and  $m_1, \ldots,m_n \in\nat$ with $m_1 < \ldots < m_n\leq k$ such that 
either $(s_1,\ldots,s_{m_1}),\ldots,(s_{m_{n-1}+1},\ldots,s_{m_n}) \in L^\xi(\Sigma, \vec{k} ; \upsilon)$ and $m_n = k$, 
or $(s_1,\ldots,s_{m_1}),\ldots,$ $(s_{m_{n-1}+1},\ldots,s_{m_n}) \in L^\xi(\Sigma, \vec{k} ; \upsilon)$,  $(s_{{m_n}+1},\ldots,s_k)\in (L^\xi(\Sigma, \vec{k} ; \upsilon))^* \setminus L^\xi(\Sigma, \vec{k} ; \upsilon)$. 
\end{prop}
\begin{proof}
It follows from the fact that every nonempty increasing sequence (finite or infinite) in $\nat$ 
has canonical representation with respect to $\A_\xi$ (cf. \cite{F3}) and that 
the family $L^\xi(\Sigma, \vec{k} ; \upsilon)$ is thin (Proposition~\ref{prop:thinfamily}).
\end{proof}

Now, using Proposition~\ref{prop:canonicalrep}, we give an alternative description of the second horn of the dichotomy described in Theorem~\ref{thm:block-Ramsey} in case the partition family is a tree.

\begin{prop}\label{prop:tree}
Let $\xi\geq 1$ be a countable ordinal, $\Sigma=\{\alpha_1, \alpha_2, \ldots \}$ be an infinite countable alphabet, $\upsilon \notin \Sigma$ a variable, $\vec{k}=(k_n)_{n\in\nat}\subseteq \nat$ an increasing sequence, $\vec{u}\in L^{\infty}(\Sigma, \vec{k}; \upsilon)$ and $\F\subseteq L^{<\infty}(\Sigma, \vec{k} ; \upsilon)$ be a tree. Then 

$L^\xi(\Sigma, \vec{k} ; \upsilon) \cap EV^{<\infty}(\vec{u})\subseteq L^{<\infty}(\Sigma, \vec{k} ; \upsilon)\setminus \F$ if and only if 

$\F\cap EV^{<\infty}(\vec{u}) \subseteq (L^\xi(\Sigma, \vec{k} ; \upsilon))^* \setminus L^\xi(\Sigma, \vec{k} ; \upsilon)$. 
\end{prop}

\begin{proof} 
Let $L^\xi(\Sigma, \vec{k} ; \upsilon) \cap EV^{<\infty}(\vec{u})\subseteq L^{<\infty}(\Sigma, \vec{k} ; \upsilon)\setminus \F$ and $\bs = (s_1,\ldots,s_k)\in \F\cap EV^{<\infty}(\vec{u})$. 
Then $\bs$ has canonical representation with respect to $L^\xi(\Sigma, \vec{k} ; \upsilon)$
(Proposition~\ref{prop:canonicalrep}), hence 
either $\bs \in (L^\xi(\Sigma, \vec{k} ; \upsilon))^* \setminus L^\xi(\Sigma, \vec{k} ; \upsilon)$, as required, or there exists $\bs_1\in L^\xi(\Sigma, \vec{k} ; \upsilon)$ such that $\bs_1 \propto \bs$. 
The second case is impossible. 
Indeed, since $\F$ is a tree and $\bs\in \F\cap EV^{<\infty}(\vec{u})$, we have 
$\bs_1 \in \F\cap EV^{<\infty}(\vec{u}) \cap L^\xi(\Sigma, \vec{k} ; \upsilon)$; a contradiction to our 
assumption. 
Hence, $\F\cap EV^{<\infty}(\vec{u}) \subseteq (L^\xi(\Sigma, \vec{k} ; \upsilon))^* \setminus L^\xi(\Sigma, \vec{k} ; \upsilon)$. 
\end{proof}

\begin{defn}\label{def:aclosed}
Let $\Sigma=\{\alpha_1, \alpha_2, \ldots \}$ be an infinite countable alphabet, ordered according to the natural numbers, $\upsilon \notin \Sigma$ a variable and $\vec{k}=(k_n)_{n\in\nat}\subseteq \nat$. We set 

$D=\{(n,\alpha): n\in \nat, \alpha \in \{\upsilon, \alpha_1, \alpha_2, \ldots, \alpha_{k_{n}} \} \}$. 
\newline
Note that $D$ is a countable set. Let $[D]^{<\omega}$ be the set of all finite subsets of $D$. Identifing every $s\in L(\Sigma, \vec{k}; \upsilon)$ with the corresponting element of $[D]^{<\omega}$ and consequently every $\bs\in L^{<\infty}(\Sigma, \vec{k}; \upsilon)$ and every $\vec{s}\in L^{\infty}(\Sigma, \vec{k}; \upsilon)$) with their characteristic functions $x_{\sigma(\bs)} \in \{0,1\}^{[D]^{<\omega}}$ and $x_{\sigma(\vec{s})} \in \{0,1\}^{[D]^{<\omega}}$ respectively, where $\sigma(\bs) = \{s_1,\ldots,s_k\}$ for every $\bs = (s_1,\ldots,s_k)\in L^{<\infty}(\Sigma, \vec{k}; \upsilon)$, $\sigma(\vec{s}) = \{s_n : n\in\nat\}$ for every $\vec{s} = (s_n)_{n\in\nat}\in L^{\infty}(\Sigma, \vec{k}; \upsilon)$ and $\sigma(\emptyset) = \emptyset$, we say that a family $\F\subseteq L^{<\infty}(\Sigma, \vec{k}; \upsilon)$ is {\em pointwise closed\/} iff the family $\{x_{\sigma(\bs)} :\bs \in \F\}$ is closed in the product topology (equivalently by the pointwise convergence topology) of $\{0,1\}^{[D]^{<\omega}}$ and in analogy a family $\U\subseteq L^{\infty}(\Sigma, \vec{k}; \upsilon)$ is {\em pointwise closed\/} iff $\{x_{\sigma(\vec{s})} :\vec{s}\in \U\}$ is closed in $\{0,1\}^{[D]^{<\omega}}$ with the product topology.
\end{defn}

\begin{prop}\label{prop:finitefamily}
Let $\Sigma=\{\alpha_1, \alpha_2, \ldots \}$ be an infinite countable alphabet, $\upsilon \notin \Sigma$ a variable and $\vec{k}=(k_n)_{n\in\nat}\subseteq \nat$.

{\rm (i)} If $\F\subseteq L^{<\infty}(\Sigma, \vec{k}; \upsilon)$ is a tree, then $\F$ is pointwise closed if and only if there does not exist an infinite sequence $(\bs_n)_{n\in\nat}$ in $\F$ 
such that $\bs_n \propto \bs_{n+1}$ and $\bs_n\ne \bs_{n+1}$ for all $n\in\nat$.

{\rm (ii)} If $\F\subseteq L^{<\infty}(\Sigma, \vec{k}; \upsilon)$  is hereditary, then $\F$ is pointwise closed if and only if there does not exist $\vec{s}\in L^{\infty}(\Sigma, \vec{k} ; \upsilon)$ such that $EV^{<\infty}(\vec{s}) \subseteq \F$. 

{\rm (iii)} The hereditary family $(L^\xi(\Sigma, \vec{k} ; \upsilon) \cap EV^{<\infty}(\vec{u}) )_*$ is pointwise closed for every countable ordinal $\xi$ and $\vec{u} \in L^{\infty}(\Sigma, \vec{k} ; \upsilon)$. 
\end{prop}
\begin{proof}
This follows directly from the definitions (for details cf. \cite{F3}, \cite{FN1}).
\end{proof}

Let $\vec{s} \in L^{\infty}(\Sigma, \vec{k} ; \upsilon)$. For a hereditary and pointwise closed family $\F\subseteq L^{<\infty}(\Sigma, \vec{k}; \upsilon)$ we will define the strong Cantor-Bendixson index $sO_{\vec{s}}(\F)$ of $\F$ with respect to $\vec{s} $.

\begin{defn}\label{def:Cantor-Bendix}
Let $\Sigma=\{\alpha_1, \alpha_2, \ldots \}$ be an infinite countable alphabet, $\upsilon \notin \Sigma$ a variable, $\vec{k}=(k_n)_{n\in\nat}\subseteq \nat$ an increasing sequence, $\vec{s}\in L^{\infty}(\Sigma, \vec{k} ; \upsilon)$ and let $\F\subseteq L^{<\infty}(\Sigma, \vec{k}; \upsilon)$ be a hereditary and pointwise closed family. For every $\xi <\omega_1$ we define the families $(\F)_{\vec{s}}^\xi$ inductively as follows:  
\newline
For every $\bw = (w_1,\ldots,w_l)\in \F\cap EV^{<\infty}(\vec{s})$ we set 

$A_{\bw}=\{t\in EV(\vec{s}) :(w_1,\ldots,w_l,t)\notin \F\}$ and $A_{\emptyset}=\{t\in EV(\vec{s}) :(t)\notin \F\} $.
\newline
We define
\begin{center}
$(\F)_{\vec{s}}^1 = \{\bw \in \F\cap EV^{<\infty}(\vec{s})\cup \{\emptyset\}
 : A_{\bw}$ does not contain an infinite orderly sequence $\}$.
\end{center}
It is easy to verify that $(\F)_{\vec{s}}^1$ is hereditary, hence pointwise closed (Proposition~\ref{prop:finitefamily}).
So, we can define for every $\xi >1$ the $\xi$-derivatives of $\F$ 
recursively as follows:
\begin{itemize}
\item[]  $ (\F)_{\vec{s}}^{\zeta +1} = ((\F)_{\vec{s}}^\zeta  )_{\vec{s}}^1$ for all $ \zeta <\omega_1$, and
\item[] $(\F)_{\vec{s}}^\xi = \bigcap_{\beta<\xi} (\F)_{\vec{s}}^\beta$ for $\xi$ a limit ordinal.
\end{itemize}

The  {\em strong Cantor-Bendixson index} $sO_{\vec{s}}(\F)$ of $\F$ on $\vec{s}$ is the smallest countable ordinal $\xi$ such that $(\F)_{\vec{s}}^\xi = \emptyset$.
\end{defn}

\begin{remark}\label{rem:Cantor-Bendix} 
Let $\vec{s}\in L^{\infty}(\Sigma, \vec{k}; \upsilon)$ and let $\F_1, \R_1,\subseteq L^{<\infty}(\Sigma, \vec{k}; \upsilon)$ be hereditary and pointwise closed families.
\begin{itemize}
\item[(i)] $sO_{\vec{s}}(\F_1)$ is a countable successor ordinal less than or 
equal to the ``usual'' Cantor-Bendixson index $O(\F_1)$ of $\F_1$ into 
$\{0,1\}^{[D]^{<\omega}}$ (cf. \cite{KM}).

\item[(ii)] $sO_{\vec{s}} (\F_1\cap EV^{<\infty}(\vec{s})) =$ $sO_{\vec{s}}(\F_1)$. 

\item[(iii)] $sO_{\vec{s}}(\F_1) \leq sO_{\vec{s}}(\R_1)$ if $\F_1\subseteq \R_1$.

\item[(iv)] If $\vec{s}_1 \prec \vec{s}$ and $\bw \in (\F_1)_{\vec{s}}^\xi$, then for every $\bw_1\in EV^{<\infty}(\vec{s}_1)$ such that $\sigma(\bw_1)=\sigma(\bw) \cap EV(\vec{s}_1) $ we have that $\bw_1\in (\F_1)_{\vec{s}_1}^\xi$ , since $EV(\vec{s}_1)\subseteq EV(\vec{s})$.

\item[(v)] If $\vec{s}_1\prec \vec{s}$, then $sO_{\vec{s}_1} (\F_1) \geq sO_{\vec{s}}(\F_1)$, 
according to (iv). 

\item[(vi)] If $\sigma(\vec{s}_1)\setminus \sigma(\vec{s})$ is a finite set, then $sO_{\vec{s}_1}(\F_1) 
\geq $$sO_{\vec{s}}(\F_1)$.
\end{itemize}
\end{remark}

The corresponding strong Cantor-Bendixson index to $L^\xi(\Sigma, \vec{k} ; \upsilon)$ is equal to $\xi+1$, with respect any sequence $\vec{s}\in L^{\infty}(\Sigma, \vec{k} ; \upsilon)$.

\begin{prop} 
\label{prop:Cantor-Bendix}
Let $\xi<\omega_1$ be an ordinal, $\Sigma=\{\alpha_1, \alpha_2, \ldots \}$ be an infinite countable alphabet, $\upsilon \notin \Sigma$ a variable, $\vec{k}=(k_n)_{n\in\nat}\subseteq \nat$ an increasing sequence and $\vec{s}\in L^{\infty}(\Sigma, \vec{k} ; \upsilon)$.
\begin{center}
 $sO_{\vec{s}_1}\Big((L^\xi(\Sigma, \vec{k} ; \upsilon) \cap EV^{<\infty}(\vec{s}))_*\Big)= \xi+1$ for every $\vec{s}_1\prec \vec{s}$.
\end{center}
\end{prop}

\begin{proof} 
The family $(L^\xi(\Sigma, \vec{k}  ; \upsilon) \cap EV^{<\infty}(\vec{s}))_*$ is hereditary and pointwise closed (Proposition~\ref{prop:finitefamily}). We will prove by induction on $\xi$ that 
$\Big((L^\xi(\Sigma, \vec{k}  ; \upsilon) \cap EV^{<\infty}(\vec{s}))_*\Big)_{\vec{s}_1}^\xi = \{\emptyset\}$ for every $\vec{s}_1\prec \vec{s}$ and $\xi<\omega_1$.
Since 
$(L^1(\Sigma, \vec{k}; \upsilon) \cap EV^{<\infty}(\vec{s}))_* = \{(t) : t\in EV(\vec{s})\} \cup \{\emptyset\}$, we have $\Big((L^1(\Sigma, \vec{k} ; \upsilon) \cap EV^{<\infty}(\vec{s}))_*\Big)_{\vec{s}_1}^1 = \{\emptyset\}$ for every $\vec{s}_1\prec \vec{s}$.

Let $\xi>1$ and assume that 
$\Big((L^\zeta (\Sigma, \vec{k} ; \upsilon) \cap EV^{<\infty}(\vec{s}))_* \Big)_{\vec{s}_1}^\zeta = \{\emptyset\}$ for every $\vec{s}_1\prec \vec{s}$ and $\zeta <\xi$. 
For every $t\in EV(\vec{s})$ with $\min dom(t)=n$, according to 
Proposition~\ref{justification}, we have 
\begin{center}
$(L^\xi(\Sigma, \vec{k} ; \upsilon)\cap EV^{<\infty}(\vec{s})) (t) = L^{\xi_n}(\Sigma, \vec{k} ; \upsilon)\cap EV^{<\infty}(\vec{s}-t)$ for some $\xi_n <\xi.$
\end{center}
Hence, for every $\vec{s}_1\prec \vec{s}$ and $t\in EV(\vec{s}_1)$ with $\min t=n$ we have that 

$\Big((L^\xi(\Sigma, \vec{k} ; \upsilon)\cap EV^{<\infty}(\vec{s}))(t)_* \Big)_{\vec{s}_1}^{\xi_n} 
= \Big((L^{\xi_n}(\Sigma, \vec{k} ; \upsilon)\cap  EV^{<\infty}(\vec{s}-t))_* \Big)_{\vec{s}_1}^{\xi_n} 
= \{\emptyset\}.$
\newline
This gives that 
$(t) \in \Big((L^\xi(\Sigma, \vec{k} ; \upsilon)\cap EV^{<\infty}(\vec{s}))_* \Big)_{\vec{s}_1}^{\xi_n}$.
So, $\emptyset \in \Big((L^\xi(\Sigma, \vec{k} ; \upsilon)\cap EV^{<\infty}(\vec{s}))_* \Big)_{\vec{s}_1}^\xi$, since if $\xi = \zeta+1$, then $(t) \in \Big((L^\xi(\Sigma, \vec{k} ; \upsilon)\cap EV^{<\infty}(\vec{s}))_* \Big)_{\vec{s}_1}^\zeta$ 
for every $t\in EV(\vec{s}_1)$ and if 
$\xi$ is a limit ordinal, then $\emptyset\in \Big((L^\xi(\Sigma, \vec{k} ; \upsilon)\cap EV^{<\infty}(\vec{s}))_* \Big)_{\vec{s}_1}^{\xi_n}$ for every $n\in\nat$ 
and $\sup \xi_n  =\xi$.

If $\{\emptyset\} \ne \Big((L^\xi(\Sigma, \vec{k} ; \upsilon)\cap EV^{<\infty}(\vec{s}))_* \Big)_{\vec{s}_1}^\xi$ for $\vec{s}_1\prec \vec{s}$, 
then there exist $\vec{s}_2\prec \vec{s}_1$ and  $s\in EV(\vec{s}_2)$ such that 
$\Big((L^\xi(\Sigma, \vec{k} ; \upsilon)\cap EV^{<\omega}(\vec{s}))(s)_* \Big)_{\vec{s}_2}^\xi \ne \emptyset$ 
(see Theorem 1.18 in \cite{F2}). 
This is a contradiction to the induction hypothesis. 
Hence, $\Big((L^\xi(\Sigma, \vec{k} ; \upsilon)\cap EV^{<\infty}(\vec{s}))_* \Big)_{\vec{s}_1}^\xi = \{\emptyset\} $ and 
$sO_{\vec{s}_1}((L^\xi(\Sigma, \vec{k} ; \upsilon) \cap EV^{<\infty}(\vec{s}))_* ) 
= \xi+1$ for every $\xi<\omega_1$.
\end{proof}

\begin{cor}\label{cor:tree} 
Let $\xi_1,\xi_2$ be countable ordinals with $\xi_1<\xi_2$ and $\vec{w}\in L^\infty (\Sigma, \vec{k} ; \upsilon)$. Then there exist $\vec{u}\prec \vec{w}$ such that  

$(L^{\xi_1}(\Sigma, \vec{k} ; \upsilon))_* \cap EV^{<\infty}(\vec{u})\subseteq (L^{\xi_2}(\Sigma, \vec{k} ; \upsilon))^* \setminus L^{\xi_2}(\Sigma, \vec{k} ; \upsilon)$.
\end{cor}

\begin{proof} 
The family $(L^{\xi_1}(\Sigma, \vec{k}; \upsilon))_* \subseteq L^{<\infty}(\Sigma, \vec{k} ; \upsilon)$ is a tree. According to Theorem~\ref{thm:block-Ramsey} and Proposition~\ref{prop:tree} there exists $\vec{u}\prec\vec{w}$ such that: 

either $L^{\xi_2}(\Sigma, \vec{k} ; \upsilon) \cap EV^{<\infty}(\vec{u})\subseteq (L^{\xi_1}(\Sigma, \vec{k} ; \upsilon))_*$, 

or $(L^{\xi_1}(\Sigma, \vec{k} ; \upsilon))_* \cap EV^{<\infty}(\vec{u}) \subseteq (L^{\xi_2}(\Sigma, \vec{k} ; \upsilon))^* \setminus L^{\xi_2}(\Sigma, \vec{k} ; \upsilon)$.
\newline
The first alternative of the dichotomy is impossible, since, according to Proposition~\ref{prop:Cantor-Bendix}
\newline
$\xi_2 +1 =$ $sO_{\vec{u}}((L^{\xi_2}(\Sigma, \vec{k} ; \upsilon) \cap EV^{<\infty}(\vec{u}))_*) \le$ $sO_{\vec{u}} ((L^{\xi_1}(\Sigma, \vec{k} ; \upsilon))_*) = \xi_1 +1$. 
\end{proof}

The following Theorem~\ref{block-NashWilliams2}, the main result in this Section, refines  Theorem~\ref{thm:block-Ramsey} in case the partition family is a tree.

\begin{defn}\label{def:blockNW} 
Let $\F \subseteq L^{<\infty}(\Sigma, \vec{k} ; \upsilon)$ be a family of finite orderly sequences of variable $\omega$-located words over an infinite countable alphabet $\Sigma$, ordered according to the natural numbers, dominated by an inccreasing sequence $\vec{k}=(k_n)_{n\in\nat}\subseteq \nat$. We set 

$\F_h = \{\bw \in \F: EV^{<\infty}(\bw)\subseteq\F \}\cup \{\emptyset\}$.

Of course, $\F_h$ is the largest subfamily of $\F\cup \{\emptyset\}$ which is hereditary.
\end{defn}
\begin{thm}
\label{block-NashWilliams2}
Let $\Sigma=\{\alpha_1, \alpha_2, \ldots \}$ be an infinite countable alphabet, $\upsilon \notin \Sigma$ a variable, $\vec{k}=(k_n)_{n\in\nat}\subseteq \nat$ an increasing sequence, $\F \subseteq L^{<\infty}(\Sigma, \vec{k} ; \upsilon)$ a family of finite orderly sequences of variable $\omega$-located words which is a tree and $\vec{w} \in L^\infty (\Sigma, \vec{k} ; \upsilon)$ an infinite orderly sequence of variable $\omega$-located words. Then we have the following cases:
\newline
\noindent {\bf [Case 1]} The family $\F_h\cap EV^{<\infty}(\vec{w})$ is not pointwise closed. 

Then, there exists $\vec{u}\prec \vec{w}$ such that 
$EV^{<\infty}(\vec{u})\subseteq \F$.
\newline
\noindent {\bf [Case 2]}
The family $\F_h \cap EV^{<\infty}(\vec{w})$ is pointwise closed.

Then, setting
\begin{center}
$\zeta_{\vec{w}}^{\F} = \sup \{$$sO_{\vec{u}} (\F_h) : \vec{u}\prec \vec{w}\}$ ,
\end{center}
which is a countable ordinal, the following subcases obtain:
\begin{itemize}
\item[2(i)] If $\xi+1 <\zeta_{\vec{w}}^{\F}$, then there exists $\vec{u}\prec \vec{w}$ 
such that 
\begin{center}
$L^\xi(\Sigma, \vec{k} ; \upsilon) \cap EV^{<\infty}(\vec{u})\subseteq \F$ ;
\end{center}
\item[2(ii)] if $\xi+1>\xi>\zeta_{\vec{w}}^{\F}$, then for every $\vec{w}_1 \prec \vec{w}$ there exists 
$\vec{u}\prec \vec{w}_1$ such that 

$L^\xi(\Sigma, \vec{k} ; \upsilon) \cap EV^{<\infty}(\vec{u})\subseteq L^{<\infty}(\Sigma, \vec{k} ; \upsilon)\setminus \F$;

(equivalently $\F\cap EV^{<\infty}(\vec{u}) \subseteq (L^\xi(\Sigma, \vec{k} ; \upsilon))^* \setminus L^\xi(\Sigma, \vec{k}; \upsilon)$) ; and  
\item[2(iii)] if $\xi+1 = \zeta_{\vec{w}}^{\F}$ or $\xi = \zeta_{\vec{w}}^{\F}$, then there exists $\vec{u}\prec \vec{w}$ such that 
\item[{}] either $L^\xi(\Sigma, \vec{k} ; \upsilon) \cap EV^{<\infty}(\vec{u})\subseteq \F\ \text{ or }\  
L^\xi(\Sigma, \vec{k} ; \upsilon) \cap EV^{<\infty}(\vec{u})\subseteq L^{<\infty}(\Sigma, \vec{k} ; \upsilon)\setminus \F.\\$ 
\end{itemize}
\end{thm}
\begin{proof}
\noindent{\bf [Case 1]}
If the hereditary family $\F_h \cap EV^{<\infty}(\vec{w})$ is not pointwise 
closed, then, according to Proposition~\ref{prop:finitefamily}, there exists $\vec{u}\in L^\infty (\Sigma, \vec{k}; \upsilon)$ such that $EV^{<\infty}(\vec{u})\subseteq \F_h\cap EV^{<\infty}(\vec{w})\subseteq \F$. Of course, $\vec{u}\prec \vec{w}$. 

\noindent{\bf [Case 2]} 
If the hereditary family $\F_h \cap EV^{<\infty}(\vec{w})$ is pointwise closed, then $\zeta_{\vec{w}}^{\F}$ 
is a countable ordinal, since the ``usual'' Cantor-Bendixson 
index $O(\F_h)$ of $\F_h$ into $\{0,1\}^{[D]^{<\omega}}$ is countable 
(Remark~\ref{rem:Cantor-Bendix}(i)) and also $sO_{\vec{u}}(\F_h)\leq O(\F_h)$ for every $\vec{u}\prec \vec{w}$.

2(i) 
Let $\xi+1<\zeta_{\vec{w}}^{\F}$. Then there exists $\vec{u}_1\prec\vec{w}$ such that 
$\xi+1 <$ $sO_{\vec{u}_1}(\F_h)$. 
According to Theorem~\ref{thm:block-Ramsey} and Proposition~\ref{prop:tree}, 
there exists $\vec{u}\prec \vec{u}_1$ such that 

either $L^\xi(\Sigma, \vec{k} ; \upsilon) \cap EV^{<\infty}(\vec{u})\subseteq \F_h\subseteq \F$, 

or $\F_h \cap EV^{<\infty}(\vec{u}) \subseteq (L^\xi(\Sigma, \vec{k} ; \upsilon))^* \setminus L^\xi(\Sigma, \vec{k} ; \upsilon)\subseteq (L^\xi(\Sigma, \vec{k} ; \upsilon))^* \subseteq (L^\xi(\Sigma, \vec{k} ; \upsilon))_*$.
\newline
The second alternative is impossible. 
Indeed, if $\F_h \cap EV^{<\infty}(\vec{u}) \subseteq (L^\xi(\Sigma, \vec{k} ; \upsilon))_*$, then, according 
to Remark ~\ref{rem:Cantor-Bendix} and Proposition~\ref{prop:Cantor-Bendix}, 

$sO_{\vec{u}_1}(\F_h) \leq$ $sO_{\vec{u}} (\F_h) =$ $sO_{\vec{u}}(\F_h\cap EV^{<\infty}(\vec{u})) \leq$ $ sO_{\vec{u}}((L^\xi(\Sigma, \vec{k} ; \upsilon))_*) = \xi+1$; 
\newline
a contradiction. Hence, $L^\xi(\Sigma, \vec{k} ; \upsilon) \cap EV^{<\infty}(\vec{u})\subseteq \F$.

2(ii)
Let $\xi +1>\xi >\zeta_{\vec{w}}^{\F}$ and $\vec{w}_1\prec \vec{w}$. According to Theorem~\ref{thm:block-Ramsey}, there exists $\vec{u_1}\prec \vec{w}_1$ such that 

 either $L^{\zeta_{\vec{w}}^{\F}}(\Sigma, \vec{k} ; \upsilon) \cap EV^{<\infty}(\vec{u}_1)\subseteq \F_h$, or  
$L^{\zeta_{\vec{w}}^{\F}}(\Sigma, \vec{k} ; \upsilon) \cap EV^{<\infty}(\vec{u_1})\subseteq L^{<\infty}(\Sigma, \vec{k} ; \upsilon)\setminus \F_h$.
\newline
The first alternative is impossible. 
Indeed, if $L^{\zeta_{\vec{w}}^{\F}}(\Sigma, \vec{k} ; \upsilon) \cap EV^{<\infty}(\vec{u}_1)\subseteq \F_h$, then, according 
to Remark ~\ref{rem:Cantor-Bendix} and Proposition~\ref{prop:Cantor-Bendix}, 
we have that 

${\zeta_{\vec{w}}^{\F}}+1 =$ $sO_{\vec{u}_1} ((L^{\zeta_{\vec{w}}^{\F}}(\Sigma, \vec{k} ; \upsilon) \cap EV^{<\infty}(\vec{u}_1))_*) \le$ $sO_{\vec{u}_1} 
(\F_h) \le \zeta_{\vec{w}}^{\F}$ ;
\newline
a contradiction. Hence, 
\begin{equation}\label{eq:block-NashWilliams} 
L^{\zeta_{\vec{w}}^{\F}}(\Sigma, \vec{k} ; \upsilon) \cap EV^{<\infty}(\vec{u}_1)\subseteq L^{<\infty}(\Sigma, \vec{k} ; \upsilon)\setminus \F_h\ .
\end{equation}
According to Theorem~\ref{thm:block-Ramsey},
there exists $\vec{u}\prec \vec{u}_1$ such that 

 either $L^\xi(\Sigma, \vec{k} ; \upsilon) \cap EV^{<\infty}(\vec{u})\subseteq \F$,
\quad or\quad  
$L^\xi(\Sigma, \vec{k} ; \upsilon) \cap EV^{<\infty}(\vec{u})\subseteq L^{<\infty}(\Sigma, \vec{k} ; \upsilon)\setminus \F$.
\newline
We claim that the first alternative does not hold. 
Indeed, if $L^\xi(\Sigma, \vec{k} ; \upsilon) \cap EV^{<\infty}(\vec{u})\subseteq \F$, then 
$(L^\xi(\Sigma, \vec{k} ; \upsilon) \cap EV^{<\infty}(\vec{u}))^* \subseteq \F^* = \F$. 
Using the canonical representation of every infinite orderly sequence of variable located words 
with respect to $L^\xi(\Sigma, \vec{k} ; \upsilon)$ (Proposition~\ref{prop:canonicalrep}) it is 
easy to check that 

$(L^\xi(\Sigma, \vec{k} ; \upsilon))^* \cap EV^{<\infty}(\vec{u}) = (L^\xi(\Sigma, \vec{k} ; \upsilon) \cap EV^{<\infty}(\vec{u}))^*$ .
\newline
Hence, $(L^\xi(\Sigma, \vec{k} ; \upsilon))^* \cap EV^{<\infty}(\vec{u})\subseteq \F$.

Since $\xi>\zeta_{\vec{w}}^{\F}$, according to Corollary~\ref{cor:tree}, there exists 
$\vec{t}\prec \vec{u}$ such that 

$(L^{\zeta_{\vec{w}}^{\F}}(\Sigma, \vec{k} ; \upsilon))_*\cap EV^{<\infty}(\vec{t}) \subseteq 
(L^\xi(\Sigma, \vec{k} ; \upsilon))^* \cap EV^{<\infty}(\vec{u})\subseteq \F$.
\newline
So, $(L^{\zeta_{\vec{w}}^{\F}}(\Sigma, \vec{k} ; \upsilon))_*\cap EV^{<\infty}(\vec{t}) \subseteq \F_h$. 
This is a contradiction to the relation \eqref{eq:block-NashWilliams}. 
Hence, $L^\xi(\Sigma, \vec{k} ; \upsilon) \cap EV^{<\infty}(\vec{u})\subseteq L^{<\infty}(\Sigma, \vec{k} ; \upsilon)\setminus \F$ and  $\F\cap EV^{<\infty}(\vec{u}) \subseteq (L^\xi(\Sigma, \vec{k} ; \upsilon))^* \setminus L^\xi(\Sigma, \vec{k} ; \upsilon)$. 

2(iii)
In the cases $\zeta_{\vec{w}}^{\F} = \xi+1$ or $\zeta_{\vec{w}}^{\F} = \xi$, we 
use Theorem~\ref{thm:block-Ramsey}.
\end{proof}

The following immediate corollary to Theorem~\ref{block-NashWilliams2} is more useful for applications. 
  
\begin{cor}
\label{cor:tree2}
Let $\Sigma=\{\alpha_1, \alpha_2, \ldots \}$, $\upsilon \notin \Sigma$, $\vec{k}=(k_n)_{n\in\nat}\subseteq \nat$ an increasing sequence, $\F \subseteq L^{<\infty}(\Sigma, \vec{k} ; \upsilon)$ which is a 
tree and let $\vec{w} \in L^\infty (\Sigma, \vec{k} ; \upsilon)$. Then
\begin{itemize}
\item[(i)] either 
there exists $\vec{u}\prec \vec{w}$ such that $EV^{<\infty}(\vec{u})\subseteq \F$,
\item[(ii)] or for every countable ordinal $\xi > \zeta_{\vec{w}}^{\F}$ there exists 
$\vec{u}\prec \vec{w}$, such that for every $\vec{u}_1 \prec \vec{u}$ the unique initial segment of $\vec{u}_1$ 
which is an element of $L^\xi(\Sigma, \vec{k} ; \upsilon)$ belongs to $ L^{<\infty}(\Sigma, \vec{k} ; \upsilon)\setminus \F$.
\end{itemize}
\end{cor}

Theorem~\ref{block-NashWilliams2} implies the following Nash-Williams type partition theorem for variable $\omega$-located words.

\begin{thm}
[\textsf{Partition theorem for infinite orderly sequences of variable $\omega$-located words}] 
\label{cor:blockNW}
Let $\Sigma=\{\alpha_1, \alpha_2, \ldots \}$ be an infinite countable alphabet, $\upsilon \notin \Sigma$ a variable, $\vec{k}=(k_n)_{n\in\nat}\subseteq \nat$ an increasing sequence. If $\U \subseteq L^{\infty}(\Sigma, \vec{k} ; \upsilon)$ is a pointwise closed family of infinite orderly sequences of variable $\omega$-located words and $\vec{w} \in L^\infty (\Sigma, \vec{k} ; \upsilon)$ an infinite orderly sequence of variable $\omega$-located words, then there exists $\vec{u}\prec \vec{w}$ such that 

either  $EV^{\infty}(\vec{u})\subseteq \U$, or $EV^{\infty}(\vec{u})\subseteq L^{\infty}(\Sigma, \vec{k} ; \upsilon) \setminus \U$.
\end{thm}

\begin{proof} 
Let $\F_{\U} = \{\bw \in L^{<\infty}(\Sigma, \vec{k} ; \upsilon)$: there exists $\vec{s}\in \U$ 
such that $\bw\propto \vec{s}\}$.
Since the family $\F_{\U}$ is a tree, we can use Corollary~\ref{cor:tree2}.
So, we have the following two cases:

\noindent { [Case 1]}
There exists $\vec{u}\prec \vec{w}$ such that $EV^{<\infty}(\vec{u})\subseteq \F_{\U}$. 
Then, $EV^{\infty}(\vec{u})\subseteq \U$.
Indeed, if $\vec{z} = (z_n)_{n\in\nat} \in EV^{\infty}(\vec{u})$, then 
$(z_1,\ldots,z_n)\in \F_{\U}$ for every $n\in\nat$. 
Hence, for each $n\in\nat$ there exists $\vec{s}_n\in\U$ such that 
$(z_1,\ldots,z_n)\propto \vec{s}_n$. 
Since $\U$ is pointwise closed, 
we have that $\vec{z}\in\U$ and consequently that $EV^{\infty}(\vec{u})\subseteq \U$.

\noindent { [Case 2]}
There exists $\vec{u}\prec \vec{w}$ such that for every $\vec{u}_1 \prec \vec{u}$ there exists an initial segment 
of $\vec{u}_1$ which belongs to $L^{<\infty}(\Sigma, \vec{k} ; \upsilon)\setminus \F_{\U}$. 
Hence, $EV^{\infty}(\vec{u})\subseteq L^{\infty}(\Sigma, \vec{k} ; \upsilon) \setminus \U$.
\end{proof}

\begin{remark}\label{rem. D8}
(i) The particular case of Theorem~\ref{cor:blockNW}, where the sequence $\vec{k}=(k_n)_{n\in\nat}\subseteq \nat$ satisfies the condition $k_n=k_1$ for every $n\in\nat$, implies Bergelson-Blass-Hindman's Theorem 6.1 in \cite{BBH}, while the case $k_n=1$ for $n\in\nat$ of Theorem~\ref{block-NashWilliams2} coincides with Theorem 4.6 in \cite{FN1}.
\end{remark}
 
We will give now a partition theorem for infinite orderly sequences of unlocated $\omega$-words, as a consequence of Theorem~\ref{cor:blockNW}. We set 

$W^\infty (\Sigma, \vec{k}; \upsilon) = \{\vec{w} = (w_n)_{n\in\nat} : w_n\in W(\Sigma, \vec{k} ; \upsilon)$ and  $w_n<w_{n+1} \ \forall\ n\in\nat\}$.
\newline
For $\vec{u}\in W^\infty (\Sigma, \vec{k}; \upsilon)$ we define

$RV^{\infty}(\vec{u}) = \{\vec{w}= (w_n)_{n\in\nat} \in W^\infty (\Sigma, \vec{k}; \upsilon): w_n\in RV(\vec{u}) \ \forall\ n\in\nat \}$.
\newline
We say that a family $\U\subseteq W^{\infty}(\Sigma, \vec{k}; \upsilon)$ is {\em pointwise closed\/} iff $\{x_{\sigma(\vec{s})} :\vec{s}\in \U\}$ is closed in $\{0,1\}^{[D]^{<\omega}}$ with the product topology, where $D=\{(n,\alpha): n\in \nat, \alpha \in \{\upsilon, \alpha_1, \alpha_2, \ldots, \alpha_{k_{n}} \} \}$, and if $\vec{s} = (s_n)_{n\in\nat}\in W^{\infty}(\Sigma, \vec{k}; \upsilon)$, then $x_{\sigma(\vec{s})} \in \{0,1\}^{[D]^{<\omega}}$ is the characteristic function of the subset $\sigma(\vec{s}) = \{s_n : n\in\nat\}$ of $[D]^{<\omega}$.

\begin{thm}
[\textsf{Partition theorem for infinite orderly sequences of variable $\omega$-words}] 
\label{cor:unblockNW}
Let $\Sigma=\{\alpha_1, \alpha_2, \ldots \}$ be an infinite countable alphabet, $\upsilon \notin \Sigma$ a variable and $\vec{k}=(k_n)_{n\in\nat}\subseteq \nat$ an increasing sequence. If $\U \subseteq W^{\infty}(\Sigma, \vec{k} ; \upsilon)$ is a pointwise closed family of infinite sequences of variable $\omega$-words and $\vec{w} \in W^\infty (\Sigma, \vec{k} ; \upsilon)$ an infinite orderly sequence of variable $\omega$-words, then there exists a reduction $\vec{u}=(u_n)_{n\in\nat}$ of $\vec{w}$ such that 

either  $RV^{\infty}(\vec{u})\subseteq \U$, or $RV^{\infty}(\vec{u})\subseteq W^{\infty}(\Sigma, \vec{k} ; \upsilon) \setminus \U$.
\end{thm}
\begin{proof}
Let $\phi:W^\infty (\Sigma, \vec{k}; \upsilon)\rightarrow RV^\infty(\vec{w})\subseteq W^{\infty}(\Sigma, \vec{k} ; \upsilon)$ be the function with $\phi( (t_n)_{n\in\nat} )=(u_n)_{n\in\nat}$, where, for $t_n=t^n_1\ldots t^n_{\lambda_{n+1}}$ for every $n\in\nat$ and $\lambda_1=0$, we define for $n\in\nat$
\begin{center}
$u_n=T_{p_{\lambda_{n}+1}}(w_{\lambda_{n}+1})\star \ldots \star T_{p_{\lambda_{n+1}}}(w_{\lambda_{n+1}})$, 
\end{center}
where for $\lambda_{n}+1\leq j\leq \lambda_{n+1}$, $p_j=0$ if $t^n_{j}=\upsilon$ and $p_j=\mu$ if $t^n_{j}=\alpha_{\mu}$ for $1\leq \mu \leq k_{j}$.
 
Let the function $f : L(\Sigma\cup\{\upsilon\}, \vec{k}) \longrightarrow W(\Sigma\cup\{\upsilon\}, \vec{k})$ which sends 
$s=s_{n_1}\ldots s_{n_l}\in L(\Sigma\cup\{\upsilon\}, \vec{k})$ to $f(s)=t_1\ldots t_{n_l}\in W(\Sigma\cup\{\upsilon\}, \vec{k})$ where $t_{n_j}=s_{n_j}$ for every $1\leq j\leq l$ and $t_i=\alpha_1$ for every $i\in\{1,\ldots, n_l\}\setminus \{n_j : 1\leq j\leq l \}$. Also, let
\begin{center}
$\tilde{f} : L^\infty (\Sigma, \vec{k} ; \upsilon)\longrightarrow 
W^\infty(\Sigma, \vec{k}; \upsilon)$ with $\hat{f}( (s_n)_{n\in\nat} )=(f(s_n))_{n\in\nat}$.  
\end{center}
The family $\hat{f}^{-1}\big(\phi^{-1}(\U)\big)\subseteq L^\infty(\Sigma, \vec{k} ; \upsilon)$ is  pointwise closed, since the functions $\phi, \hat{f}$ are continuous. So, according to Corollary~\ref{cor:blockNW}, there exists $\vec{s}=(s_n)_{n\in\nat}\in L^\infty (\Sigma, \vec{k} ; \upsilon)$ such that 
either  $EV^{\infty}(\vec{s})\subseteq \hat{f}^{-1}\big(\phi^{-1}(\U)\big)$, or $EV^{\infty}(\vec{u})\subseteq L^{\infty}(\Sigma, \vec{k} ; \upsilon) \setminus \hat{f}^{-1}\big(\phi^{-1}(\U)\big)$.
\newline
Set $t_n=f(s_n)$ for every $n\in\nat$ and $\vec{t}=(t_n)_{n\in\nat}\in W^{\infty}(\Sigma, \vec{k} ; \upsilon)$. Then 
\begin{center}
either $RV^{\infty}(\vec{t})\subseteq (\phi)^{-1}(\U)$, or $RV^{\infty}(\vec{t})\subseteq W^\infty (\Sigma, \vec{r}; \upsilon)\setminus(\phi)^{-1}(\U)$. 
\end{center}
Let $t_n=t^n_{1}\ldots t^n_{q_n}$ for every $n\in\nat$. Set $\lambda_1=0$ and 
$u_n=T_{p_{\lambda_{n}+1}}(w_{\lambda_{n}+1})\star \ldots \star T_{p_{\lambda_{n+1}}}(w_{\lambda_{n+1}})$ for every $n\in\nat$, 
where, for $\lambda_{n}+1\leq j \leq \lambda_{n+1}$, $p_j=0$ if $t^n_{j}=\upsilon$ and $p_j=\mu$ if $t^n_{j}=\alpha_{\mu}$ for some $1\leq \mu \leq k_{j}$. Then $\vec{u}=(u_n)_{n\in\nat}$ is a reduction of $\vec{w}$ such that 

either  $RV^{\infty}(\vec{u})\subseteq \U$, or $RV^{\infty}(\vec{u})\subseteq W^{\infty}(\Sigma, \vec{k} ; \upsilon) \setminus \U$.  
\end{proof}

\begin{remark}\label{rem. g}
(i) Carlson's partition theorem (Theorem 2 in \cite{C}) is the particular case of Theorem~\ref{cor:unblockNW} where the sequence $\vec{k}=(k_n)_{n\in\nat}\subseteq \nat$ satisfies the condition $k_n=k_1$ for every $n\in\nat$. 
\end{remark}

As a consequence of Theorem~\ref{cor:blockNW} we will state and prove partition theorems for the infinite sequences in a commutative and in a noncommutative semigroup, respectively, which are simultaneous extensions of the infinitary partition theorem of Milliken \cite{M}, Taylor \cite{T} and of van der Waerden theorem \cite{vdW} for semigroups. For a sequence $(x_n)_{n\in\nat}$ in a semigroup $(X,+)$, we set
\begin{center}
$\big[FS\big((x_n)_{n\in\nat}\big)\big]^\omega = \{(y_n)_{n\in\nat} : y_n\in FS\big((x_n)_{n\in\nat}\big)$ and $y_n<y_{n+1}$ for every $n\in\nat\}$.
\end{center}
For $y=x_{n_1}+\ldots+x_{n_\lambda}, z=x_{m_1}+\ldots+x_{m_\nu}\in FS\big((x_n)_{n\in\nat}\big)$ we write $y<z$ if $n_\lambda<m_1$.  For a set $X$ let $X^{\nat}=\{(x_n)_{n\in\nat} : x_n\in X\}$. We endow the set $X^{\nat}$ with the product topology (equivalently by the pointwise convergence topology).

\begin{thm}
\label{cor:centralNW}
Let $(X,+)$ be a commutative semigroup, $\vec{k}=(k_n)_{n \in \nat}\subseteq \nat$ an increasing sequence and $(y_{l,n})_{n \in \nat}$ for every $l \in \nat,$ sequences in $X$. If $\U \subseteq X^{\nat}$ is a pointwise closed family of $X^{\nat}$, then there exist sequences $(E_n)_{n \in \nat}\subseteq [\nat]^{<\omega}_{>0},$ with $E_n<E_{n+1}$ for every $n \in \nat$ and $(H_n)_{n \in \nat}\subseteq [\nat]^{<\omega}_{>0},$  with $H_n\subseteq E_{n}$ for every $n \in \nat$ and a sequence $(\beta_n)_{n \in \nat}\subseteq X$ with $\beta_n=\sum_{j\in E_n \setminus H_n}y_{l^n_j,j}$ for $1\leq l^n_j \leq k_j$ such that for every function $f: \nat\rightarrow \nat$ with $f(n)\leq k_n$ for every $n\in\nat$
\newline
either  $\big[ FS\big((\beta_{n}+\sum_{t \in H_{n}}y_{f(n), t})_{n\in\nat}\big)\big]^\omega\subseteq \U$, 
or $\big[ FS\big((\beta_{n}+\sum_{t \in H_{n}}y_{f(n), t})_{n\in\nat}\big)\big]^\omega\subseteq X^{\nat}\setminus\U$.
\end{thm}
\begin{proof}
Let the alphabet $\Sigma=\{1,2,\ldots\}=\nat$, $\upsilon=0$ and $x_0\in X$. We set $y_{0,\;n}=x_0$ for every $n\in\nat$ and let the function $g:L(\Sigma,\vec{k}; \upsilon)\rightarrow X$ with
\begin{center}
$g(w_{n_1}\ldots w_{n_l})=\sum^{l}_{i=1}y_{w_{n_i},\;n_i}$ for $w_{n_1}\ldots w_{n_l}\in L(\Sigma,\vec{k}; \upsilon).$ 
\end{center}
Let $\hat{g} : L^\infty (\Sigma, \vec{k} ; \upsilon)\longrightarrow X^{\nat}$ with 
$\hat{g}( (w_n)_{n\in\nat} )=(g(w_n))_{n\in\nat}$.  
\newline
The family $\hat{g}^{-1}(\U)\subseteq L^\infty(\Sigma, \vec{k} ; \upsilon)$ is  pointwise closed, since the functions $\hat{g}$ is continuous. So, according to Corollary~\ref{cor:blockNW}, there exists $\vec{w}=(w_n)_{n\in\nat}\in L^\infty (\Sigma, \vec{k} ; \upsilon)$ such that 
\begin{center}
either  $EV^{\infty}(\vec{w})\subseteq \hat{g}^{-1}(\U)$, or $EV^{\infty}(\vec{w})\subseteq L^{\infty}(\Sigma, \vec{k} ; \upsilon) \setminus \hat{g}^{-1}(\U)$.
\end{center}
Let $w_n=w^{n}_{q^{n}_1}\ldots w^{n}_{q^{n}_{l_n}}$ for every $n \in \nat$. We can suppose that $w^{n}_{q^{n}_1}, w^{n}_{q^{n}_{l_n}}\in \Sigma$ for every $n \in \nat$. Otherwise replace the sequence $(w_n)_{n\in\nat}$ by the sequence $(u_n)_{n\in\nat}$, where $u_n=T_1(w_{3n-1})\star w_{3n}\star T_1(w_{3n+1})$. We set 
$E_n =\{q^{n}_i:\; 1\leq i\leq l_n\}$ the domain of $w_n$, 
$H_n=\{q^{n}_i \in E_n: w^{n}_{q^{n}_{i}}=\upsilon \}$ and 
$\beta_n=\sum_{j\in E_n\setminus H_n}y_{w^n_j,j}\in X$.
\end{proof}
Analogously, can be proved a partition theorem for the infinite sequences in a noncommutative semigroup. 
\begin{thm}
\label{cor:ncentralNW}
Let $(X,+)$ be a noncommutative semigroup, $\vec{k}=(k_n)_{n \in \nat}\subseteq \nat$ an increasing sequence and $(y_{l,n})_{n \in \nat}$ for every $l \in \nat,$ sequences in $X$. If $\U \subseteq X^{\nat}$ is a pointwise closed family of $X^{\nat}$, then there exist sequences $(E_n)_{n \in \nat}\subseteq [\nat]^{<\omega}_{>0},$ with $E_n<E_{n+1}$ for every $n \in \nat$, $(m_n)_{n\in\nat}$, $E^n_1<\cdots <E^n_{m_n+1}$ such that $E_n=E^n_1\cup\ldots \cup E^n_{m_n+1}$ for every $n\in\nat$, $H^n_1,\ldots, H^n_{m_n}$ with $H^n_i \subseteq E^n_i$ for every $1\leq i\leq m_n$, $n \in \nat$ and $\beta^n_{1},\ldots,\beta^n_{m_n+1}\subseteq X$ with $\beta^n_i=\sum_{j\in E^n_i \setminus H^n_i}y_{l^n_j,j}$ with $1\leq l^n_j \leq k_j$ if $1\leq i\leq m_n$ and 
$\beta^n_{m_n+1}=\sum_{j\in E_{m_n+1}}y_{l^n_j,j}$ with $1\leq l^n_j \leq k_j$ for every $n \in \nat$, such that for every function $f: \nat\rightarrow \nat$ with $f(n)\leq k_n$ for every $n\in\nat$

either  $\Big[ FS\Big(\big( (\sum^{m_{n}}_{i=1}(\beta^{n}_{i}+\sum_{t \in H^{n}_{i}}y_{f(n),\;t}))+\beta^{n}_{m_{n}+1}\big)_{n\in\nat}\Big)\Big]^\omega\subseteq \U$, 

or $\Big[ FS\Big(\big( (\sum^{m_{n}}_{i=1}(\beta^{n}_{i}+\sum_{t \in H^{n}_{i}}y_{f(n),\;t}))+\beta^{n}_{m_{n}+1}\big)_{n\in\nat}\Big)\Big]^\omega\subseteq X^{\nat}\setminus\U$.
\end{thm}

\section{The characterization of Ramsey partitions of infinite sequences of variable $\omega$-located words }

As a consequence of Theorem~\ref{block-NashWilliams2} we prove, in Theorem~\ref{thm:Ellentuck} below, a partition theorem for infinite sequences of variable $\omega$-located words stronger than Theorem~\ref{cor:blockNW} involving a topology $\Tau_E$ on the space $L^{\infty}(\Sigma, \vec{k} ; \upsilon)$ stronger than the relative topology of 
$L^{\infty}(\Sigma, \vec{k} ; \upsilon)$ with respect to the product topology of $\{0,1\}^{[D]^{<\omega}}$. 
A consequence of Theorem~\ref{thm:Ellentuck} (together with 
Corollary~\ref{cor:Ellentuck2}) is the characterization of Ramsey partitions of $L^{\infty}(\Sigma, \vec{k} ; \upsilon)$ in terms of the Baire property 
in the topology $\Tau_E$ (in Theorem~\ref{cor:Baireproperty}). Using Theorem~\ref{cor:Baireproperty} can be proved an analogous characterization of the Ramsey partitions of the set $W^{\infty}(\Sigma, \vec{k} ; \upsilon)$ of all the sequences of variable $\omega$-words. 

We will define below the topology $\Tau_E$ on the space $L^{\infty}(\Sigma, \vec{k}; \upsilon)$, an 
analogue of the Ellentuck topology on $\nat$ defined in \cite{E}.

\begin{defn}\label{def:Ellentuck}
Let $\Sigma=\{\alpha_1, \alpha_2, \ldots \}$ be an infinite countable alphabet, $\upsilon \notin \Sigma$ a variable and $\vec{k}=(k_n)_{n\in\nat}\subseteq \nat$ an increasing sequence.  We define the topology $\Tau_E$ on $L^{\infty}(\Sigma, \vec{k} ; \upsilon)$ as the topology with basic open sets of 
the form:
\begin{equation*}
[\bt ,\vec{s}] =  \{\vec{w} \in L^{\infty}(\Sigma, \vec{k} ; \upsilon): \bt \propto \vec{w}\quad\text{and}
\quad \vec{w} - \bt \prec \vec{s}\}\ ,
\end{equation*}
where $ \bt \in L^{<\infty}(\Sigma, \vec{k} ; \upsilon)$ and
$\vec{s}\in L^{\infty}(\Sigma, \vec{k} ; \upsilon)$.

The topology $\Tau_E$ is stronger than the relative topology of 
$L^{\infty}(\Sigma, \vec{k} ; \upsilon)$ with respect to the product topology of $\{0,1\}^{[D]^{<\omega}}$, which has basic open sets of the form 

$[\bt,\vec{e}] = \{\vec{w} \in L^{\infty}(\Sigma, \vec{k} ; \upsilon): \bt\propto \vec{w}\}$ 
\newline
where $\vec{e}=(e_n)_{n\in\nat}$ with $e_n : \{n\}\rightarrow \{\upsilon\}$ for every $n\in \nat$. 

We denote by $\hat\U$ and $\U^\lozenge$ the closure and the interior 
respectively of a family $\U\subseteq L^{\infty}(\Sigma, \vec{k} ; \upsilon)$ in the topology 
$\Tau_E$. 
Then it is easy to see that 
\begin{equation*}
\begin{split}
&\hat\U = \{\vec{w}\in L^{\infty}(\Sigma, \vec{k} ; \upsilon): [\bt ,\vec{w}]\cap \U\ne\emptyset
\ \text{ for every }\  \bt\propto \vec{w}\}\ ;\ \text{ and}\\
&\U^\lozenge = \{\vec{w}\in L^{\infty}(\Sigma, \vec{k} ; \upsilon):\ \text{ there exists }
\bt\propto \vec{w}\ \text{ such that }\ 
[\bt,\vec{w}] \subseteq \U\}\ .
\end{split}
\end{equation*}
\end{defn}

If $ \bt = (t_1,\ldots,t_k) \in L^{<\infty}(\Sigma, \vec{k} ; \upsilon)$ and  $ \bs = (s_1,\ldots,s_l) \in L^{<\infty}(\Sigma, \vec{k} ; \upsilon)$ with  $t_k < s_1$, then we set $ \bt \odot \bs = (t_1,\ldots,t_k,s_1,\ldots,s_l) \in L^{<\infty}(\Sigma, \vec{k} ; \upsilon)$.      

\begin{thm}\label{thm:Ellentuck}
Let $\Sigma=\{\alpha_1, \alpha_2, \ldots \}$ be an infinite countable alphabet, $\upsilon \notin \Sigma$ a variable, $\vec{k}=(k_n)_{n\in\nat}\subseteq \nat$ an increasing sequence,  $\U\subseteq L^{\infty}(\Sigma, \vec{k} ; \upsilon)$, $\bt\in L^{<\infty}(\Sigma, \vec{k} ; \upsilon)$ and $\vec{w}\in L^{\infty}(\Sigma, \vec{k} ; \upsilon)$. 
Then
\begin{itemize}
\item[{}] either there exists $\vec{u}\prec \vec{w}$ such that $[\bt,\vec{u}]\subseteq \hat\U$,
\item[{}] or there exists a countable ordinal $\xi_0= \zeta_{(\bs,\vec{w})}^{\U}$ 
such that for every $\xi>\xi_0$ there exists $\vec{u}\prec \vec{w}-\bt$ with 
$[\bt \odot \bs,\vec{u}]\subseteq L^{\omega}(\Sigma, \vec{k} ; \upsilon)\setminus \U$ for every 
$\bs \in L^\xi(\Sigma, \vec{k} ; \upsilon) \cap EV^{<\infty}(\vec{u})$.
\end{itemize}
\end{thm}

We will give the proof after the following lemma which 
is analogous to Lemma~\ref{lem:block-Ramsey}. 

\begin{lem}\label{lem:Ellentuck}
Let $\R \subseteq\{[\bt,\vec{s}] :\bt\in L^{<\infty}(\Sigma, \vec{k} ; \upsilon)$ and 
$\vec{s}\in L^{\infty}(\Sigma, \vec{k} ; \upsilon)\}$ with the properties:

(i) for every $(\bt,\vec{s})\in L^{<\infty}(\Sigma, \vec{k} ; \upsilon)\times L^{\infty}(\Sigma, \vec{k} ; \upsilon)$ there exists $\vec{s}_1\prec \vec{s}$ such that $[\bt,\vec{s}_1]\in\R$; and 

(ii) for every $[\bt,\vec{s}]\in\R$ and $\vec{s}_1\prec \vec{s}$ we have 
$[\bt,\vec{s}_1] \in\R$.
\newline
Then, for every $(\bt,\vec{w})\in L^{<\infty}(\Sigma, \vec{k} ; \upsilon)\times L^{\infty}(\Sigma, \vec{k} ; \upsilon)$ there exists $\vec{s}\in [\bt,\vec{w}]$ such that $[\bt \odot \bs,\vec{u}]\in\R$ 
for every $\bs\in EV^{<\infty}(\vec{s} - \bt)$ and $\vec{u}\prec \vec{s} - \bt$.
\end{lem}

\begin{proof}
Let $ \bt = (t_1,\ldots,t_k) \in L^{<\infty}(\Sigma, \vec{k} ; \upsilon)$ and $\vec{w}\in L^{\infty}(\Sigma, \vec{k} ; \upsilon)$.
We can assume that $\vec{w} - \bt = \vec{w}$. 
According to the assumption~(i), there exists $\vec{s}_1\prec \vec{w}$ such that 
$[\bt,\vec{s}_1] \in\R$. 
Assume that $\vec{s}_n\prec \cdots \prec \vec{s}_1 \in L^{\infty}(\Sigma, \vec{k} ; \upsilon)$ have been constructed and 
$\vec{s}_n = (s_i^n)_{i\in\nat}$ for every $n\in\nat$.

Set $\{\bs_1,\ldots,\bs_\lambda\} = EV^{<\infty}((s_1^1,\ldots, s_n^n))$. 
According to (i), there exist $\vec{s}_{n+1}^1 \prec \vec{s}_n - s_n^n$  such that  
$[\bt \odot \bs_1,\vec{s}_{n+1}^1]\in\R$,  $\vec{s}_{n+1}^2 \prec \vec{s}_{n+1}^1$  such that 
$[\bt \odot \bs_2,\vec{s}_{n+1}^2]\in \R$, and finally $\vec{s}_{n+1}^\lambda \prec \vec{s}_{n+1}^{\lambda-1}\prec \vec{s}_n-s_n^n$ 
such that  $[\bt \odot \bs_\lambda,\vec{s}_{n+1}^\lambda]\in \R$. 
Set $\vec{s}_{n+1} = \vec{s}_{n+1}^\lambda = (s_i^{n+1})_{i\in\nat}$. 
Then, according to (ii), $[\bt \odot \bs_i,\vec{s}_{n+1}]\in\R$ for every 
$1\le i\le \lambda$. 

Set $\vec{s} = (t_1,\ldots,t_k,s_1^1,s_2^2,\ldots)\in L^{\infty}(\Sigma, \vec{k} ; \upsilon)$. 
Then $\vec{s} \in [\bt,\vec{w}]$.
Let $\bs\in EV^{<\infty}(\vec{s} - \bt)$ with $\bs\ne\emptyset$. 
If $n_0 = \min \{n\in\nat: \bs\in EV^{<\infty}((s_1^1,\ldots,s_n^n))\}$,
then $[\bt \odot \bs,\vec{s}_{n_0+1}]\in \R$. Let $\vec{u}\prec \vec{s} - \bt$. Then 
$[\bt \odot \bs,\vec{u}] = [\bt \odot \bs,\vec{u} - s_{n_0}^{n_0}]\in \R$, according to assumption~(ii). 
If $\bs =\emptyset$, then $[\bt,\vec{s}_1]\in\R$, hence $[\bt,\vec{u}] \in \R$.
\end{proof}

\begin{proof}[Proof of Theorem~\ref{thm:Ellentuck}] 
Let $\U\subseteq L^{\infty}(\Sigma, \vec{k} ; \upsilon)$, $\bt\in L^{<\infty}(\Sigma, \vec{k} ; \upsilon)$ and $\vec{w}\in L^{\infty}(\Sigma, \vec{k} ; \upsilon)$. 
Set 
\begin{equation*}
\begin{split}
\R_{\U} = &\{[\bw,\vec{s}] :(\bw,\vec{s}) \in L^{<\infty}(\Sigma, \vec{k} ; \upsilon)\times L^{\infty}(\Sigma, \vec{k} ; \upsilon)
\ \text{ and}\\
&\qquad \text{either }\ [\bw,\vec{s}]\cap \U = \emptyset\quad\text{or}\quad 
[\bw,\vec{s}_1]\cap \U\ne\emptyset\ \text{ for every }\ \vec{s}_1\prec \vec{s}\}\ .
\end{split}
\end{equation*}
It is easy to check that $\R_{\U}$ satisfies the assumptions~(i) and (ii) 
of Lemma~\ref{lem:Ellentuck}, 
hence, there exists $\vec{w}_1\in [\bt,\vec{w}]$ such that $[\bt \odot \bs,\vec{w}_1]\in\R_{\U}$ 
for every $\bs\in EV^{<\infty}(\vec{w}_1 - \bt)$. Set
\begin{equation*}
\F = \{\bs\in EV^{<\infty}(\vec{w}_1 - \bt) : [\bt \odot \bs ,\vec{w}_2] \cap \U\ne \emptyset\text{ for every } \vec{w}_2\prec \vec{w}_1\}\ .
\end{equation*}
The family $\F$ is a tree. 
Indeed, let $\bs \in \F$ and $\bs_1 \propto\bs$.
Then $[\bt \odot \bs_1,\vec{w}_1] \in\R_{\U}$, since $\bs_1 \in EV^{<\infty}(\vec{w}_1 - \bt)$.
So, either $[\bt \odot \bs_1,\vec{w}_1]\cap \U=\emptyset$, which is impossible, since 
$[\bt \odot \bs,\vec{w}_1]\cap \U\ne\emptyset$, or $[\bt \odot \bs_1,\vec{w}_2]\cap \U
\ne\emptyset$ for every $\vec{w}_2\prec \vec{w}_1$. 
Hence, $\bs_1 \in \F$. 

We use Theorem~\ref{block-NashWilliams2} for $\F$ and 
$\vec{w}_1 - \bt$.
We have the following cases:

\noindent {\bf [Case 1]} 
There exists $\vec{u}\prec \vec{w_1} - \bt \prec \vec{w}$ such that 
$EV^{<\infty}(\vec{u}) \subseteq \F$. 
This gives that $[\bt \odot \bs,\vec{u}_1]\cap \U\ne \emptyset$ for every 
$\bs \in EV^{<\infty}(\vec{u})$ and $\vec{u}_1 \prec \vec{u}$, 
which implies that $[\bt,\vec{u}] \subseteq\hat\U$.

\noindent {\bf [Case 2]} 
There exists a countable ordinal $\xi_0 = \zeta_{(\bs,\vec{w})}^{\U}$ such that for every $\xi>\xi_0$ there exists 
$\vec{u}\prec \vec{w}_1-\bt\prec \vec{w}-\bt$ with 
$L^\xi(\Sigma, \vec{k} ; \upsilon) \cap EV^{<\infty}(\vec{u}) \subseteq L^{<\infty}(\Sigma, \vec{k} ; \upsilon)\setminus\F$.
Then $[\bt * \bs,\vec{u}] \subseteq L^{\infty}(\Sigma, \vec{k} ; \upsilon)\setminus \U$ for every 
$\bs \in L^\xi(\Sigma, \vec{k} ; \upsilon) \cap EV^{<\infty}(\vec{u})$.
\end{proof}

Applying Theorem~\ref{thm:Ellentuck} to partitions $\U$ that are closed 
(the pointwise closed families are included in this class) or meager in the 
topology $\Tau_E$, we consider the following consequences.

\begin{cor}\label{cor:Ellentuck}
Let $\U$ be a closed subset of $L^{\infty}(\Sigma, \vec{k} ; \upsilon)$ for the topology $\Tau_E$, $\bt \in L^{<\infty}(\Sigma, \vec{k} ; \upsilon)$ and $\vec{w}\in L^{\infty}(\Sigma, \vec{k} ; \upsilon)$. 
Then 

 either there exists $\vec{u}\prec\vec{w}$ such that $[\bt,\vec{u}]\subseteq \U$, 

 or there exists a countable ordinal $\xi_0 = \zeta_{(\bs,\vec{w})}^{\U}$, 
such that for every $\xi>\xi_0$ there exists $\vec{u}\prec \vec{w}-\bt$ such that 
$[\bt \odot \bs,\vec{u}]\subseteq L^{\infty}(\Sigma, \vec{k} ; \upsilon)\setminus \U$ for every 
$\bs \in L^\xi(\Sigma, \vec{k} ; \upsilon) \cap EV^{<\infty}(\vec{u})$.
\end{cor}

\begin{cor}\label{cor:Ellentuck2}
Let $\U$ be a subset  of $L^{\infty}(\Sigma, \vec{k} ; \upsilon)$ meager in the topology $\Tau_E$, 
$\bt \in L^{<\infty}(\Sigma, \vec{k} ; \upsilon)$ and $\vec{w}\in L^{\infty}(\Sigma, \vec{k} ; \upsilon)$. 
Then, there exists a countable ordinal $\xi_0$ such that for every $\xi>\xi_0$
there exists $\vec{u}\prec \vec{w}-\bt$ such that $[\bt \odot \bs,\vec{u}]\subseteq L^{\infty}(\Sigma, \vec{k} ; \upsilon)\setminus \U$ for every 
$\bs \in L^\xi(\Sigma, \vec{k} ; \upsilon) \cap EV^{<\infty}(\vec{u})$.
\end{cor}

\begin{proof}
We apply Theorem~\ref{thm:Ellentuck} for $\U$.
We will prove that the first alternative of the dichotomy is impossible. 
Indeed, let $\vec{u}\prec \vec{w}$ such that $[\bt,\vec{u}]\subseteq \hat\U$.  
If $\U= \bigcup_{n\in\nat} \U_n$ with $(\hat\U_n)^\lozenge = \emptyset$ 
for every $n\in\nat$, then we set 
\begin{equation*}
\begin{split}
\R = &\{[\bw,\vec{s}] : \bw \in L^{<\infty}(\Sigma, \vec{k} ; \upsilon),\ \vec{s}\in L^{\infty}(\Sigma, \vec{k} ; \upsilon)
\ \text{ and}\\
&\qquad [\bw,\vec{s}]\cap \U_k =\emptyset\ \text{ for every } \ k\in\nat
\text{ with } k\le |\bw|\}\ ;
\end{split}
\end{equation*}
where $|\bw|$ denotes the cardinality of the set $\sigma(\bw)$. 

The family $\R$ satisfies the conditions (i) and (ii) of 
Lemma~\ref{lem:Ellentuck}. 
Indeed, according to Theorem~\ref{thm:Ellentuck}, for every $\bw\in L^{<\infty}(\Sigma, \vec{k} ; \upsilon)$, $\vec{s}\in L^{\infty}(\Sigma, \vec{k} ; \upsilon)$ and $k\in\nat$ there exists 
$\vec{s}_1\prec \vec{s}$ such that $[\bw,\vec{s}_1]\cap \U_k = \emptyset$, as 
it is impossible $[\bw,\vec{s}_1]\subseteq \hat\U_k$. Thus $\R$ satisfies (i) and 
obviously satisfies (ii). 
Hence, there exists 
$\vec{u}_1\in[\bt,\vec{u}]$ such that 
$[\bt \odot \bs,\vec{u}_1]\in\R$ for every $\bs\in EV^{<\infty}(\vec{u}_1-\bt)$. 
Then, $[\bt,\vec{u}_1]\cap \U=\emptyset$. 
Indeed, let $\vec{u}_2\in [\bt,\vec{u}_1]\cap\U$. 
Then, $\vec{u}_2 \in [\bt,\vec{u}_1] \cap \U_k$ for some $k\in\nat$. 
Hence, there exists $\bs\in EV^{<\infty}(\vec{u}_1-\bt)$ with $\bt\ \odot \bs\propto \vec{u}_2$,
$k\le |\bt \odot \bs|$ and $[\bt \odot \bs,\vec{u}_1]\cap \U_k\ne\emptyset$.
Then, $[\bt \odot \bs,\vec{u}_1]\notin \R$. 
A contradiction, since $\bs\in EV^{<\infty}(\vec{u}_1-\bt)$. 
Thus, $[\bt,\vec{u}_1]\cap \U=\emptyset$ and consequently $\vec{u}_1\notin \hat\U$. But 
$\vec{u}_1 \in [\bt ,\vec{u}] \subseteq\hat\U$, a contradiction. 
Hence, the second alternative of Theorem~\ref{thm:Ellentuck} holds for $\U$. 
\end{proof}

\begin{defn}\label{def:Milliken}
A family $\U\subseteq L^{\infty}(\Sigma, \vec{k} ; \upsilon)$ of infinite orderly sequences of variable $\omega$-located words is called {\em completely 
Ramsey} if for every $\bt\in L^{<\infty}(\Sigma, \vec{k} ; \upsilon)$ and every $\vec{w}\in L^{\infty}(\Sigma, \vec{k} ; \upsilon)$ there exists $\vec{u}\prec \vec{w}$ such that 
\begin{equation*}
\text{either }\ [\bt,\vec{u}]\subseteq \U\quad\text{ or }\quad 
[\bt,\vec{u}]\subseteq L^{\infty}(\Sigma, \vec{k} ; \upsilon) \setminus \U\ .
\end{equation*}
\end{defn}

A consequence of Theorem~\ref{thm:Ellentuck} is the characterization of completely Ramsey families of infinite orderly sequences of variable $\omega$-located words.

\begin{thm}\label{cor:Baireproperty}
Let $\Sigma=\{\alpha_1, \alpha_2, \ldots \}$ be an infinite countable alphabet, $\upsilon \notin \Sigma$ a variable and $\vec{k}=(k_n)_{n\in\nat}\subseteq \nat$ an increasing sequence. A family $\U\subseteq L^{\infty}(\Sigma, \vec{k} ; \upsilon)$ is completely Ramsey if and only if 
$\U$ has the Baire property in the topology $\Tau_E$.
\end{thm} 

\begin{proof} 
Let $\U\subseteq L^{\infty}(\Sigma, \vec{k} ; \upsilon)$ have the Baire property in the topology 
$\Tau_E$. 
Then $\U = \B \triangle  \C  = (\B \cap \C^c) \cup (\C\cap \B^c)$, where 
$\B\subseteq L^{\infty}(\Sigma, \vec{k} ; \upsilon)$ is $\Tau_E$-closed and 
$\C\subseteq L^{\infty}(\Sigma, \vec{k} ; \upsilon)$ is $\Tau_E$-meager $(\C^c= L^{\infty}(\Sigma, \vec{k} ; \upsilon)\setminus\C)$.  
Let $\bt\in L^{<\infty}(\Sigma, \vec{k} ; \upsilon)$ and $\vec{w}\in L^{\infty}(\Sigma, \vec{k} ; \upsilon)$. 
According to Corollary~\ref{cor:Ellentuck} and Proposition~\ref{prop:canonicalrep}, there exists $\vec{u}_1 \prec\vec{w}$ such that $[\bt,\vec{u}_1]\subseteq \C^c$ and according to Corollary~\ref{cor:Ellentuck2} there exists $\vec{u}\prec \vec{u}_1$ such that 

 either $[\bt,\vec{u}] \subseteq \B\cap [\bt,\vec{u}_1]\subseteq 
\B\cap \C^c \subseteq \U$
 or $[\bt,\vec{u}]\subseteq \B^c \cap [\bt,\vec{u}_1] \subseteq 
\B^c \cap \C^c \subseteq \U^c$.
\newline
Hence, $\U$ is completely Ramsey.

On the other hand, if $\U$ is completely Ramsey, then $\U = \U^\lozenge
\cup (\U\setminus \U^\lozenge)$ and $\U\setminus\U^\lozenge$ is a meager set 
in $\Tau_E$. 
Hence $\U$ has the Baire property in the topology $\Tau_E$.
\end{proof}

\section*{Acknowledgments}
Author acknowledge partial support from the Kapodistrias research grant of Athens University.

\bigskip
{\footnotesize
\noindent 
\newline
Address of the author: 
\newline
{\sc Department of Mathematics, Athens University, Panepistemiopolis, 15784 Athens, Greece}
\newline  
E-mail address: vfarmaki@math.uoa.gr

\end{document}